\documentclass[final]{siamart220329}
\usepackage{hyperref}
\usepackage{epsfig}
\usepackage{amsfonts}
\usepackage{algorithm}
\usepackage{algorithmic}
\usepackage{graphicx}
\usepackage{subfigure}
\usepackage{array}
\usepackage{enumerate}
\usepackage{enumitem}
\usepackage{rotating}
\usepackage[T1]{fontenc}
\AtBeginDocument{
    \hypersetup{colorlinks,urlcolor=black}
}

\usepackage{pgfplots,tikz}

\newdimen\prevdp
\def\leftlabel#1{\noalign{\prevdp=\prevdepth
		\kern-\prevdp\nointerlineskip\vbox to0pt{\vss\hbox{#1}}\kern\prevdp}}

\newcommand {\C}        {{\mathbb{C}}}
\newcommand {\R}        {{\mathbb{R}}}

\newcommand{\norm}[1]{\left\Vert#1\right\Vert}
\newcommand{\abs}[1]{\left\vert#1\right\vert}

\newcommand {\cL}        {{\mathcal{L}}}
\newcommand {\cV}        {{\mathcal{V}}}

\newtheorem{assumption}[theorem]{Assumption}

\newcommand {\ri}	{{\mathrm i}}

\DeclareMathOperator{\Col}{Col}

\DeclareMathOperator{\Real}{Re}
\DeclareMathOperator{\Imag}{Im}

\DeclareMathOperator{\orth}{orth}

\usepackage{todonotes}

\headers{N. Aliyev and E. Mengi}{Large-Scale Pseudospectral Abscissa Minimization}

\title{Large-Scale Minimization of the Pseudospectral Abscissa}
\author{
Nicat Aliyev\thanks{Czech Technical University, Department of Instrumentation and Control, Technicka 4, 16607, Prague, Czech Republic  \href{mailto:nijat.aliyev@cvut.cz}{(nijat.aliyevt@gmail.com)}.} \and 
Emre Mengi\thanks{Ko\c{c} University, Department of Mathematics, Rumeli Feneri Yolu 34450, Sar{\i}yer, Istanbul, Turkey
\href{mailto:emengi@ku.edu.tr}{(emengi@ku.edu.tr)}.}}
\date{\today}

\begin{document}
\maketitle

\begin{abstract}
\noindent
This work concerns the minimization of the pseudospectral abscissa of a matrix-valued function
dependent on parameters analytically. The problem is motivated by robust stability and 
transient behavior considerations for a linear control system that has optimization
parameters. We describe a subspace procedure to cope with the setting when the matrix-valued
function is of large size. The proposed subspace procedure solves a sequence of reduced problems
obtained by restricting the matrix-valued function to small subspaces, whose dimensions increase
gradually. It possesses desirable features such as a superlinear convergence
exhibited by the decay in the errors of the minimizers of the reduced problems. 
In mathematical terms, the problem we consider is a large-scale
nonconvex minimax eigenvalue optimization problem such that the eigenvalue function
appears in the constraint of the inner maximization problem. Devising and analyzing 
a subspace framework for the minimax eigenvalue optimization problem at hand with the eigenvalue
function in the constraint require special treatment that makes use of a Lagrangian and dual variables.
There are notable advantages in minimizing the pseudospectral abscissa over maximizing the
distance to instability or minimizing the $\mathcal{H}_\infty$ norm; the optimized pseudospectral
abscissa provides quantitative information about the worst-case transient growth, and the initial
guesses for the parameter values to optimize the pseudospectral abscissa can be arbitrary, unlike the case 
to optimize the distance to instability and $\mathcal{H}_\infty$ norm that would normally require initial 
guesses yielding asymptotically stable systems.

\end{abstract}



\medskip

\section{Introduction}\label{sec:introduction}
The minimization of the spectral abscissa of a linear control system has drawn interest
in the last couple of decades \cite{Burke2002, Apkarian2006}. A classical problem that can
be tackled using the spectral abscissa minimization is the stabilization by static output feedback (SOF) problem; 
given matrices $A \in {\mathbb C}^{n\times n}$, $B \in {\mathbb C}^{n\times m}$, $C\in {\mathbb C}^{p\times n}$, 
find a controller $K \in {\mathbb C}^{m\times p}$ such that the system $x'(t) = (A + BKC) x(t)$ is 
asymptotically stable, equivalently $A + BKC$ has all of its eigenvalues on the open left-half of
the complex plane. SOF is known to be a notoriously difficult problem \cite{Blondel1995}. Indeed,
it has been shown that SOF when the entries of $K$ are subject to box constraints is NP-hard
\cite{Blondel1997, Nemirovski1993}. Mathematically, the spectral abscissa minimization is a nonconvex
eigenvalue optimization problem that involves the minimization of the real part of the rightmost
eigenvalue. Nonsmoothness at a locally optimal point minimizing the spectral abscissa can 
occur due to possible nonsimplicity of the rightmost eigenvalue at the local optimizer \cite{Burke2001, Burke2001a},
as well as due to the existence of multiple rightmost eigenvalues with the same real part; see, e.g., 
the Turbo generator example in \cite{Burke2003b}.
Numerically speaking a bigger challenge is the non-Lipschitz nature of the rightmost eigenvalue at 
a local optimizer; when the rightmost eigenvalue is not simple at a local optimizer,
it can change rapidly near the optimizer. In more formal terms, the spectral abscissa does not
have to be Lipschitz continuous, not even locally, at a local optimizer, and this causes numerical
difficulties to numerical algorithms, which are at least prone to rounding errors. A second difficulty
with minimizing the spectral abscissa is that, even if a negative spectral abscissa guarantees an asymptotic
decay, the system can still exhibit significant transient growth before the eventual decay.

As a remedy to these problems with the spectral abscissa minimization, the pseudospectral abscissa
minimization has been considered in the last two decades \cite{Burke2003b,Apkarian2006b}.
Recall that the spectral abscissa of a matrix $A \in {\mathbb C}^{n\times n}$ is given by
$\alpha(A) := \max \{ \text{Re}(z) \; | \; z \in \Lambda(A) \}$, where $\Lambda(\cdot)$ denotes the spectrum
(i.e., the set of all eigenvalues) of its matrix argument. 
On the other hand, for a given real number $\epsilon>0$, the \emph{$\epsilon$-pseudospectrum} 
of $A$, which we denote by $\Lambda_{\epsilon}(A)$, 
consists of eigenvalues of all matrices within an $\epsilon$-neighborhood of $A$, formally defined as
\begin{equation*}
	\Lambda_{\epsilon}(A)
			\;	:=	\;
	\left\{
			z	\in	\C 
				\; | \; 
		z \in \Lambda(A+\Delta) \;\; 
				\exists \Delta\in\C^{n \times n} \; \text{ s.t. } \norm{\Delta}_2 \leq \epsilon \right\},
\end{equation*}
and the \emph{$\epsilon$-pseudospectral abscissa} of $A$ is defined as 
$\alpha_{\epsilon}(A) := \max\left\{\Real(z) \; | \; z\in \Lambda_{\epsilon}(A) \right\}$,
that is as the real part of the rightmost point in $\Lambda_{\epsilon}(A)$ \cite{Trefethen1999, Trefethen2005}.

Unlike the spectral abscissa, the pseudospectral map $A \mapsto \alpha_{\epsilon}(A)$ is locally Lipschitz 
continuous \cite[Corollary 7.2]{Lewis2008}, \cite[Corollary 3.4]{Gurbuzbalaban2012}. 
Moreover, if $\alpha_{\epsilon}(A)$ is negative, the system at hand as well as all nearby
systems at a distance of $\epsilon$ are asymptotically stable. 
This is significant especially if the entries of the matrix are uncertain.
A second virtue in using the pseudospectral
abscissa rather than the spectral abscissa is that, if it is sufficiently small, not only the asymptotic decay
but also a nice transient behavior of the solution of the autonomous system $x'(t) = A x(t)$ 
may be possible, thanks to the Kreiss-matrix theorem \cite{Trefethen2005, Kreiss1962}. 
Equivalent and computationally plausible 
characterizations for $\Lambda_\epsilon(A)$ and $\alpha_\epsilon(A)$ are given by \cite{Trefethen2005}
\begin{equation}\label{eq:psa_chars}
\begin{split}
	& \Lambda_{\epsilon}(A)	\;	=	\;	\left\{z\in\C \; | \; \sigma_{\min}(A-zI)\leq \epsilon \right\},	\\[.1em]
	& \alpha_{\epsilon}(A)	\;	=	\;	\max\left\{\Real(z) \; | \; z \in {\mathbb C}  \;\: \text{s.t.} 	\;\:
														\sigma_{\min}(A-zI)\leq \epsilon \right\},
\end{split}
\end{equation}
where $\sigma_{\min}(\cdot)$ denotes the smallest singular value of its matrix argument.
We remark that $\alpha_{\epsilon}(A)$, just like the spectral abscissa,
is usually a nonsmooth function of the entries of $A$
when the rightmost point in $\Lambda_{\epsilon}(A)$ is not unique, or the singular
value $\sigma_{\min}(A-\widetilde{z}I)$ is not simple at a rightmost point $\widetilde{z}$
of $\Lambda_{\epsilon}(A)$.

The computation of the $\epsilon$-pseudospectral abscissa and its derivatives require more work compared
to that for the spectral abscissa, yet there are very good algorithms to compute
the $\epsilon$-pseudospectral abscissa. The criss-cross algorithm developed by
Burke, Lewis and Overton is globally convergent at a quadratic rate, hence computes
the pseudospectral abscissa very reliably and efficiently for small- to medium-size matrices 
\cite{Burke2003, Overton2003}. For the pseudospectral abscissa of larger-size matrices, 
there is a fixed-point iteration developed by Guglielmi et al. \cite{Guglielmi2011}, which can
further be accelerated with the subspace framework in \cite{Kressner2014}.
We also refer to \cite{Benner2019} for improvements on the quadratically convergent
criss-cross algorithm, and its extension to spectral value sets.

\subsection{Problem and Contributions}
In this work, 
we assume we are given a matrix-valued function  $A : \Omega \rightarrow \C^{n \times n} $ of the form
\begin{equation}\label{eq:affine_mvf}
	A(x) = f_1(x)A_1+\dots+f_{\kappa}(x)A_{\kappa}
\end{equation}
dependent on the parameters $x$,
where $A_1,\dots,A_{\kappa}\in\C^{n \times n}$, the size of the matrices $n$ is very large, and 
$f_1,\dots,f_{\kappa}:\Omega\rightarrow \R$ are real-analytic on $\Omega$, a nonempty, open subset of $\R^d$
representing the permissible values for the parameters. We deal with the $\epsilon$-pseudospectral abscissa
minimization problem, for a prescribed $\epsilon > 0$, which is the minimax problem
\begin{equation}\label{min_full}
	\begin{split}
		\min_{x\in\underline{\Omega}} \: \alpha_{\epsilon}(A(x)) 
			& \;\; = \;\; \min_{x\in\underline{\Omega}} \: \max\left\{\Real(z) \; | \;   z\in \Lambda_{\epsilon}(A(x)) \right\} \\
			&\;\; = \;\; \min_{x\in\underline{\Omega}} \: \max\left\{\Real(z) \;  |  \;  
								\; z \in {\mathbb C}  \;\: \text{s.t.} 	\;\:   \sigma_{\min}(A(x)-zI)\leq \epsilon \right\}
	\end{split}
\end{equation}
over a compact, convex subset $\underline{\Omega}$ of $\Omega$ with nonempty interior. 
We describe a subspace procedure 
that reduces the size of $A$ considerably but without altering the optimal parameter values, and argue
that the proposed subspace procedure possesses desirable convergence properties in theory such as global convergence
and a superlinear rate of convergence. As a result, it enables us to solve the $\epsilon$-pseudospectral
abscissa minimization problems involving parameter dependent matrices with sizes on the order of
thousands.

The stabilization by static output feedback problem
for given $A\in\C^{n\times n}$, $B\in\C^{n\times m}$, $C\in\C^{p\times n}$ can be treated in a robust way
by  minimizing $\alpha_{\epsilon}(A + BKC)$ 
over $K \in \C^{m\times p}$ with entries constrained to lie in prescribed intervals. 
If the minimal value of $\alpha_{\epsilon}(A + BKC)$ is negative, not only the system but also nearby systems
are stabilizable.
Such a stabilization
problem falls into the setting of (\ref{min_full}), as $A + BKC$ can be represented in the form (\ref{eq:affine_mvf}). 
Variants of Newton's Method, especially BFGS, with proper line-search have been successfully
applied to such nonsmooth problems recently. The difficulty is that these smooth optimization techniques 
converge at a linear rate at best on nonsmooth problems. They would require the computation of
$\alpha_{\epsilon}(A + BKC)$ in the objective quite a few times, which makes them prohibitively
expensive in the large-scale setting. The package HANSO  \cite{Hanso} is based on a hybrid method that
makes use of both BFGS and a gradient sampling algorithm \cite{Burke2002}.
A specialization of this hybrid method for ${\mathcal H}_\infty$ controller design 
called HIFOO \cite{Burke2006, Gumussoy2009, Arzelier2011}
is applicable for the minimization of the spectral abscissa and pseudospectral
abscissa, as well as maximizing the distance to instability of $A + BKC$ over $K$.  
The more recent package GRANSO \cite{Curtis2017} is also based on BFGS, and is 
a variant of HANSO that can cope with box constraints. The downsides of all of these
methods are that they converge locally, and they are not well-suited for large-scale problems.
Locally-convergent bundling techniques and spectral bundle methods 
\cite{Apkarian2006b, Apkarian2008, Apkarian2006} are also used for solving problems 
related to SOF.

The problem at hand is a nonconvex large-scale minimax eigenvalue optimization problem. What makes
it peculiar compared to our previous works \cite{Mengi2018, Aliyev2020} is that the eigenvalue function
appears in the constraint of the inner maximization problem. This is in contrast to \cite{Mengi2018} and \cite{Aliyev2020} 
that introduce subspace frameworks for large-scale minimax eigenvalue optimization problems -
specifically for the maximization of the distance to instability and minimization of the ${\mathcal H}_\infty$-norm,
respectively - where the eigenvalue functions appear in the objective. Designing a subspace framework
for a minimax problem with the eigenvalue function in the constraint, and analyzing its convergence require 
a special treatment. For instance, when establishing the superlinear convergence of the proposed framework,
we work on the Lagrangian as well as the dual variable as much as the primal variables. 

Minimizing the $\epsilon$-pseudospectral abscissa of $A(x)$, that is the problem in (\ref{min_full}), 
and maximizing the distance to instability of $A(x)$ (more generally minimizing the
${\mathcal H}_\infty$-norm for a linear time invariant system depending on parameters) 
are motivated by similar robust stability and transient behavior considerations. 
However, there are advantages in 
optimizing the pseudospectral abscissa. First, by minimizing the pseudospectral abscissa, 
we simultaneously minimize a concrete lower bound on the largest transient growth possible. 
Secondly, when maximizing the distance to instability
or minimizing the ${\mathcal H}_\infty$ norm, the system with the initial guess 
for the parameter values should 
ideally be asymptotically stable, and finding such a guess may be a challenge. When 
minimizing the pseudospectral abscissa, it does not matter to start with 
parameter values leading to systems not asymptotically stable.

\subsection{Outline}
We introduce the subspace framework in the next section, then investigate its properties such
as the derivatives of the original and reduced problems, and, consequently, deduce Hermite interpolation
properties between the original and reduced problems. The proposed framework produces a
sequence of reduced problems with sizes increasing gradually. In Section \ref{sec:rate_conv},
we show that if the minimizers of the reduced problems stagnate,
then the point of stagnation is actually a global minimizer of the original problem, and
provide a rate-of-convergence analysis of the proposed framework. 
In this section, under mild assumptions, we prove a superlinear convergence result for the errors of the reduced problems.
Section \ref{sec:r_psa} discusses the extensions of the proposed framework to minimize the
real $\epsilon$-pseudospectral abscissa, the real part of the rightmost point
in the real $\epsilon$-pseudospectrum when the perturbations are constrained to be real matrices.
A Matlab implementation of the proposed framework is made publicly available. In Section \ref{sec:num_results},
we perform numerical experiments with this implementation on synthetic as well as benchmark examples
from the \emph{COMP}$l_e ib$ collection \cite{Leibfritz2004}, and observe that the deduced theoretical 
global convergence and superlinear convergence results hold in practice. The numerical results illustrate 
the efficiency and accuracy of the subspace framework in practice on large-scale pseudospectral 
abscissa minimization problems.

\section{Subspace Framework}\label{sec:sf}
In this section, we present a subspace framework for the minimization of the pseudospectral abscissa of a 
large-scale parameter dependent matrix $A(x)$ of the form (\ref{eq:affine_mvf}).
We resort to one-sided projections to deal with the large size of $A(x)$. Specifically, 
given a subspace ${\mathcal V} \subseteq \mathbb{C}^n$  
of dimension $k \ll n,$ and a matrix $V \in {\mathbb C}^{n\times k}$ whose columns form an orthonormal basis 
for ${\mathcal V}$, we minimize the pseudospectral abscissa of the reduced matrix-valued function
\begin{equation}\label{eq:red_mvf}
	A^V(x) := A(x)V =  f_1(x)A_1 V +\dots+f_{\kappa}(x)A_{\kappa} V
\end{equation}
instead of minimizing the pseudospectral abscissa of $A(x)$.
Formally, we define the $\epsilon$-pseudospectrum of the reduced matrix-valued function by
\begin{equation*}
	\begin{split}
	\Lambda_\epsilon (A^V(x)) \;
		 := \;
			\left\{  z \in \mathbb{C} \;\; | \;\; \sigma_{\min} (A^V(x)-zV)\leq \epsilon
			 \right\}	.
	\end{split}
\end{equation*}
A significant difference compared to the square case is that the rectangular matrix pencil 
$L(\lambda) = A^V(x) - \lambda V$ does not have to have an eigenvalue, i.e., it is possible
that $(A^V(x) - \lambda V) w \neq 0$ for every nonzero vector $w \in {\mathbb C}^k$ for 
every $\lambda \in {\mathbb C}$ \cite{Stewart1994, Gantmacher1959}. A consequence is
it may turn out to be the case that for $\epsilon > 0$ sufficiently small 
$\sigma_{\min}(A^V(x) - z V) > \epsilon$ for all $z \in {\mathbb C}$, 
equivalently $\Lambda_\epsilon (A^V(x)) =\emptyset $. Ruling this possibility out, 
under the assumption that $\Lambda_\epsilon (A^V(x)) \neq \emptyset$,
we define the $\epsilon$-pseudospectral abscissa of the reduced matrix-valued function in (\ref{eq:red_mvf}) by
\begin{equation}\label{eq:rec_psa}
	\begin{split}
	\alpha_\epsilon (A^V(x)) \;
		 				& :=  \;	\max \left\{  \text{Re}(z) \;\; | \; \; z\in \Lambda_\epsilon (A^V(x)) \right\} 		
						\\
				& \phantom{:}=  \;  \max \left\{  \text{Re}(z) \;\; | \;\; z \in {\mathbb C} \,\, \text{ s.t. } \, \sigma_{\min}( A^V(x) - zV) \leq \epsilon  \right\},
	\end{split}
\end{equation} 
and, under the assumption that $\Lambda_\epsilon (A^V(x)) \neq \emptyset$ for all $x \in \underline{\Omega}$, 
solve the minimization problem
\begin{equation}\label{min_pseudo_spectral}
	\min_{x\in\underline\Omega} \: \alpha_{\epsilon}(A^V(x)) \;\: = \;\:  \min_{x\in \underline\Omega} \max \left\{  \text{Re}(z) \;\; | \; \; z \in {\mathbb C} \, \text{ s.t. } \, \sigma_{\min}( A^V(x) - zV) \leq \epsilon  \right\}
\end{equation}
rather than (\ref{min_full}).

Employing two-sided projections for the reduction of $A(x)$ may appear as a plausible strategy since $A(x)$ is a 
non-Hermitian matrix-valued function, and since two-sided projections lead to square matrices with 
pseudospectra guaranteed to be nonempty. Yet, two-sided projections cause convergence problems, 
such as the loss of a monotonicity property discussed in the next section, crucial to set up an interpolatory
framework that converges globally and quickly.

We remark that $\Lambda_\epsilon (A^V(x))$ is independent 
of the choice of the orthonormal basis used for ${\mathcal V}$, that is 
$\Lambda_{\epsilon}(A^{V_1}(x)) = \Lambda_{\epsilon}(A^{V_2}(x))$  
for two different matrices $V_1$, $V_2$ whose columns form orthonormal bases 
for ${\mathcal V}$, as $\sigma_{\min}( A^{V_1}(x) - z V_1 ) = \sigma_{\min}( A^{V_2} (x) - zV_2)$
for any $z\in \C$ for such $V_1, V_2$. 
Accordingly, letting $\cV:=\Col(V)$, and hiding the dependence of the 
pseudospectra on $A(x)$, we use the shorthand notations
\begin{equation*}
	\Lambda_\epsilon (x) :=  \Lambda_\epsilon (A(x)), 
						\quad \Lambda^{\mathcal V}_\epsilon (x):= \Lambda_\epsilon (A^V(x)) \, .
\end{equation*}	
It follows that, when $\Lambda^{\mathcal V}_\epsilon (x) = \Lambda_\epsilon (A^V(x)) \neq \emptyset$, 
the $\epsilon$-pseudospectral abscissa $\alpha_\epsilon (A^V(x))$ is independent of the choice of
the orthonormal basis, so we also use the shorthands			
\begin{equation*}
	\alpha_\epsilon (x):= \alpha_\epsilon (A(x)), 
						\quad\, \alpha^{\mathcal V}_\epsilon (x):= \alpha_\epsilon (A^V(x)) \, 		
\end{equation*}
throughout the rest of this text.

The basic subspace framework for the minimization of $\alpha_{\epsilon}(x)$ is described 
in Algorithm \ref{alg1}.  Throughout, we always keep the assumption below regarding Algorithm \ref{alg1},
even though we do not state it explicitly. 
\begin{assumption}\label{ass:nonempty_psa}
The subspace ${\mathcal V}_1$ generated by Algorithm \ref{alg1} satisfies
$\Lambda_\epsilon^{{\mathcal V}_1}(x) \neq \emptyset$ for all $x \in \underline{\Omega}$. 
\end{assumption}
Under Assumption \ref{ass:nonempty_psa},
the monotonicity result in Section \ref{sec:basic_results}, i.e., part \textbf{(ii)} of Lemma \ref{monotonicity}, implies 
$\Lambda^{{\mathcal V}_k}_\epsilon(x) \neq \emptyset$ for all $x \in \underline{\Omega}$ for all 
$k \geq 2$ as well.
A justification for Assumption \ref{ass:nonempty_psa} is given towards the end of Section \ref{sec:basic_results};
see in particular Theorem \ref{thm:notempty_psa} and its implications. It turns out choosing
$x^{(1)}_1, \dots x^{(1)}_\eta$ in line \ref{alg:init_points} of Algorithm \ref{alg1} on a sufficiently fine 
uniform grid for $\underline{\Omega}$ typically guarantees the satisfaction of Assumption \ref{ass:nonempty_psa}. 
Moreover, letting $L := \inf_{z \in {\mathbb C}} \sigma_{\min}( A^{V_1}(x) - zV_1)$, for any $\epsilon > L$
we have $\Lambda^{{\mathcal V}_1}_{\epsilon}(x) \neq \emptyset$. Thus, provided $\epsilon > 0$ is large enough, 
$\Lambda^{{\mathcal V}_1}_{\epsilon}(x) \neq \emptyset$ for all $x \in \underline{\Omega}$. That is if $\epsilon$ is
large enough, Assumption \ref{ass:nonempty_psa} is again guaranteed to be satisfied.
\begin{algorithm}
\begin{algorithmic}[1]
	\REQUIRE{The matrix-valued function $A(x)$, the feasible region $\underline{\Omega} \,$, and $\epsilon > 0$.}
	\ENSURE{An estimate $\widehat{x}$ for $\arg\min_{x\in\underline{\Omega}}\alpha_\epsilon (x)$, and $\widehat{z} \in {\mathbb C}$
				that is an estimate for a globally rightmost point in $\Lambda_{\epsilon}(\widehat{x})$}	
	\vskip .6ex
	\STATE{$x^{(1)}_1, \dots x^{(1)}_\eta \gets$ initially chosen points in $\underline{\Omega}$.}\label{alg:init_points}
	\STATE{$z^{(1)}_j \gets \arg\max\left\{\text{Re}(z) \; | \; z \in {\mathbb C} \, \text{ s.t. } \,  \sigma_{\min}(A(x^{(1)}_j)-zI)\leq \epsilon \right\}$ for $j = 1, \dots , \eta$.}\label{line:abs1}
	\STATE{$v^{(1)}_{j} \gets$ a right singular vector corr. to $\sigma_{\min}(A(x^{(1)}_j)-z^{(1)}_jI)$ for $j = 1, \dots, \eta$.} \label{line:sv1}	
	\STATE{$\mathcal{V}_1 \gets$ span$\left\{v^{(1)}_{1} , \dots v^{(1)}_{\eta}\right\} \quad$ and $\quad V_1 \gets$ an orthonormal basis for ${\mathcal V}_1$.}\label{defn:V0} 
	\vskip .4ex
	\FOR{$k=2,3,\dots$}
	\vskip .3ex
	\STATE{$x^{(k)}\gets\arg\min_{x\in\underline{\Omega}}\alpha_\epsilon^{{\mathcal V}_{k-1}} (x)$.} \label{siter_start}
	\vskip .3ex
	\STATE{$z^{(k)} \gets \arg\max\left\{\text{Re}(z) \; |  \; z \in {\mathbb C} \, \text{ s.t. } \, 
										\sigma_{\min}(A(x^{(k)})-zI)\leq \epsilon \right\}$.}\label{line:absk}
	\vskip .4ex
	\STATE{\textbf{Return} $\widehat{x} \gets x^{(k)}$, $\widehat{z} \gets z^{(k)}$ if convergence occurred.} \label{alg:terminate}
	\vskip .75ex
	\STATE{$v^{(k)} \gets$ a unit right singular vector corresponding to $\sigma_{\min}(A(x^{(k)})-z^{(k)}I)$.}\label{line:svk} 
	\vskip .6ex
	\STATE
	{$V_k\gets \orth\left( [V_{k-1} \;\; v^{(k)}]\right) \quad$ {and}  
	$\quad \mathcal{V}_k \gets \Col(V_k)$.} \label{siter_end}
	\vskip .5ex
	\ENDFOR
\end{algorithmic}
\caption{The subspace framework to minimize $\alpha_{\epsilon}(x)$ over $\underline{\Omega}$}
\label{alg1}
\end{algorithm}

At each subspace iteration in Algorithm \ref{alg1} in lines \ref{siter_start}--\ref{siter_end}, 
first a small-scale reduced problem is solved in line \ref{siter_start}, in particular, a global minimizer $\widetilde{x}$ 
is found for a reduced problem. Then a rightmost point $\widetilde{z}$ of $\Lambda_{\epsilon}(\widetilde{x})$
is determined in line \ref{line:absk}. Finally, the subspace is expanded with the inclusion of a right singular vector 
corresponding to $\sigma_{\min}(A(\widetilde{x}) - \widetilde{z} I)$ in line \ref{siter_end}.
Note that the optimization problems in lines \ref{line:abs1}, \ref{siter_start} and \ref{line:absk}
are nonconvex, and can have more than one global optimizers; argmin and argmax in these lines
refer to any global minimizer and any global maximizer.

Determining the rightmost point of $\Lambda_{\epsilon}(\widetilde{x})$ involves the large-scale matrix-valued
function $A(x)$, and is usually the most expensive step computationally in a subspace iteration.
For this task, we usually benefit in practice from the approach in \cite{Kressner2014}, an approach that combines the
globally-convergent criss-cross algorithm \cite{Overton2003} for computing the pseudospectral
abscissa with a subspace framework. On the other hand, the small-scale reduced pseudspectral abscissa 
minimization problem in line \ref{siter_start} is usually cheap to solve. For this task, we employ  ``eigopt'', the globally convergent algorithm 
in \cite{Mengi2014}, if there is only one parameter, or otherwise, if the matrix-valued function depends on multiple parameters, 
we use ``GRANSO''\,\cite{Curtis2017}. The latter approach ``GRANSO'' is based on an adaptation
of BFGS with a proper line search for nonsmooth optimization problems, so can converge to local
minimizers rather than global minimizers, yet better suited for problems with several parameters.
In contrast, ``eigopt'' based on piecewise quadratic approximations of the objective
is guaranteed to converge globally if a lower bound on the second derivatives is
provided accurately as a parameter, but prohibitively expensive when there are more than
a few parameters.  
Last but not the least, let us also note that the right singular vector of the large-scale 
matrix in line \ref{line:svk} should normally be calculated iteratively, for instance by means of ``ARPACK''
\cite{Lehoucq1998}. 
Further implementation details of the subspace framework, including the condition to check convergence
used in practice in line \ref{alg:terminate}, are described in Section \ref{sec:num_implement}.

\subsection{Basic Results Regarding the Subspace Framework}\label{sec:basic_results}
Next, we present two basic results concerning the subspace framework that will be crucial 
in the convergence analysis, followed by a result that supports Assumption \ref{ass:nonempty_psa},
that is the assumption $\Lambda^{{\mathcal V}_1}_{\epsilon}(x) \neq \emptyset$ for all 
$x \in \underline{\Omega}$ regarding Algorithm \ref{alg1}.

The first result is the monotonicity property with respect to the subspace $\mathcal{V}$. 
We refer to Lemma 3.1 and succeeding arguments in \cite{Kressner2014} for a proof.

\begin{lemma}[Monotonicity]\label{monotonicity}
Let $\, \mathcal{V},\mathcal{W}$ be two subspaces of $\, \C^n$ such that 
$\mathcal{V} \subseteq \mathcal{W},$ and $V,W$ be matrices 
whose columns form orthonormal bases for $\, \mathcal{V}, \mathcal{W}$. 
Then the following assertions hold:
\begin{enumerate}
\item[\bf (i)]
$\sigma_{\min}(A(x)-zI) \, \leq \, \sigma_{\min}(A^{{W}}(x)-zW) 
				\, \leq \, 
\sigma_{\min}(A^{{V}}(x)-zV) \;\; \forall z\in\C, \, \forall x\in\Omega$.

\item[\bf (ii)]
$\Lambda_{\epsilon}^{\mathcal{V}}(x) \, \subseteq \, \Lambda_{\epsilon}^{\mathcal{W}}(x) \, \subseteq \, \Lambda_{\epsilon}(x)	
						\;\;\;		\forall x\in\Omega$.

\item[\bf (iii)]
$\alpha_{\epsilon}^{\mathcal{V}}(x) \, \leq \, \alpha_{\epsilon}^{\mathcal{W}}(x) \,  \leq  \, \alpha_{\epsilon}(x)  \;\;\; \forall x\in\Omega \;$ such that $\; \Lambda_{\epsilon}^{\mathcal V}(x) \neq \emptyset$.
\end{enumerate}
\end{lemma}
The next result concerns the interpolation properties between the full and reduced problems, and
is borrowed from \cite[Lemma 3.2]{Kressner2014}, where a proof is also provided.
\begin{lemma}\label{lemma:interpol} 
For a given $\widetilde{x} \in {\Omega}$, a given subspace $\mathcal{V}$, the following are equivalent:
\begin{enumerate}
	\item[\bf (i)] $\Lambda^{\mathcal V}_\epsilon (\widetilde{x}) \neq \emptyset$ and
	$\alpha^{\mathcal V}_\epsilon (\widetilde{x}) = \alpha_\epsilon (\widetilde{x})$.
	\item[\bf (ii)] The subspace  $\mathcal{V}$ contains a right singular vector corresponding
	to $\sigma_{\min} (A(\widetilde{x})-\widetilde{z}I)$ for some $\widetilde{z} \in \Lambda_{\epsilon}(\widetilde{x})$ with
	$\mathrm{Re}(\widetilde{z}) = \alpha_\epsilon (\widetilde{x})$.
\end{enumerate}
\end{lemma}

The two results above pave the way for the conclusion that $\Lambda_{\epsilon}^{\mathcal V}(\widetilde{x}) \neq \emptyset$
at a given $\widetilde{x} \in \Omega$ provided proper vectors are included in the subspace ${\mathcal V}$. 
The next result has a stronger assertion. Its proof is included 
in Appendix \ref{app:nonempty_psa_proof}. In this result and in the subsequent sections, 
${\mathcal B}(\widetilde{c}, r)$ for a given $\widetilde{c} \in {\mathbb R}^q$ and $r \in {\mathbb R}$ denotes the closed ball
\[
	{\mathcal B}(\widetilde{c}, r)
		\;	:=	\;
	\{ \,	c \in {\mathbb R}^q	\;\;	|	\;\;	\| c - \widetilde{c} \|_2 \; \leq \; r	 \, \},
\]
which reduces to the closed interval $[\widetilde{c}-r,\widetilde{c}+r]$ in case $q = 1$. 
\begin{theorem}\label{thm:notempty_psa}
Let $\widetilde{x} \in \underline{\Omega}$, and suppose ${\mathcal V}$ is a subspace containing $v$, which is a 
right singular vector corresponding
to $\sigma_{\min} (A(\widetilde{x})-\widetilde{z}I)$ for some $\widetilde{z} \in \Lambda_{\epsilon}(\widetilde{x})$ with
$\mathrm{Re}(\widetilde{z}) = \alpha_\epsilon (\widetilde{x})$. Then 
$\widetilde{z} \in \Lambda^{\mathcal V}_{\epsilon}(\widetilde{x})$. Moreover,
if the singular value $\sigma_{\min} (A(\widetilde{x})-\widetilde{z}I)$ is simple,
there exists $\nu \in {\mathbb R}$, $\nu > 0$ independent of the subspace ${\mathcal V}$ such that
\begin{equation}\label{eq:nonempty_ball}
	\Lambda^{\mathcal V}_{\epsilon}(x) 	\;	 \neq 	\;	\emptyset		\quad\quad	
		\forall x \in {\mathcal B}(\widetilde{x}, \nu) \cap \Omega	\:	.
\end{equation}
Additionally, suppose for every $x \in \underline{\Omega}$ there is a corresponding rightmost point $z$
in $\Lambda_{\epsilon}(x)$ such that $\sigma_{\min} (A(x) - zI)$ is simple.
Then $\nu$ in (\ref{eq:nonempty_ball}) is also independent of $\widetilde{x} \in \underline{\Omega}$.
\end{theorem}
An implication of Theorem \ref{thm:notempty_psa} regarding Algorithm \ref{alg1} is that,
under the simplicity assumption on the smallest singular value at the rightmost points
of $\epsilon$-pseudospectra,
if $x^{(1)}_1, \dots , x^{(1)}_\eta$ are grid points on a sufficiently fine uniform grid for $\underline{\Omega}$,
then
$\Lambda^{{\mathcal V}_1}_\epsilon(x) \neq \emptyset$ for all $x \in \underline{\Omega}$. 
Hence, Assumption \ref{ass:nonempty_psa} holds.

Global convergence of the proposed subspace framework as well as its rapid convergence can be attributed to 
the interpolation properties between the full and the reduced problems. In a subsequent subsection, we establish 
interpolation properties between $\alpha_\epsilon(\cdot)$ and $\alpha_\epsilon^{\cV_k}(\cdot)$ as well as between their 
first derivatives for the subspaces ${\mathcal V}_k$ generated by Algorithm \ref{alg1}.
Before stating the result formally, we first focus on the formulas for the derivatives of $\alpha_\epsilon(\cdot)$
and $\alpha_\epsilon^{\cV}(\cdot)$ for a given subspace $\cV$ in the next subsection.

\subsection{Derivatives of the Pseudospectral Abscissa}\label{sec:psa_derivatives}
We view the pseudospectral abscissa as a constrained optimization problem.
Formally, letting $A(x,z): = A(x)-(z_1 + \ri z_2)I$ and $\sigma(x,z):=\sigma_{\min}(A(x,z))$ for 
$z= (z_1,z_2)\in\R^2$,  the $\epsilon$-pseudospectral abscissa $\alpha_\epsilon(x)$ can 
be expressed as
\begin{equation}\label{eq:opt1}
	\alpha_{\epsilon}(x) 		\;	=	\;	
			\max \{ z_1	\;\;		|	\;\; 	 z =  (z_1,z_2)\in\R^2, \;\;
								 \sigma(x, z )- \epsilon \leq 0	\}.						
\end{equation}
We consider the Langrangian function
\begin{equation}\label{eq:Lagrange}
	\mathcal{L} (x,z,\mu)		\;	:=		\;	 z_1-\mu(\sigma(x,z)-\epsilon)  
\end{equation}
associated with \eqref{eq:opt1}, where $\mu \geq 0$ is the Lagrange multiplier corresponding 
to the constraint $\sigma(x, z )- \epsilon \leq 0$. 
For a given $x$, we denote the optimal $z = (z_1 , z_2)$ and the corresponding $\mu$ for 
the optimization problem in (\ref{eq:opt1}), assuming they are unique, with $z(x) = (z_1(x), z_2(x))$ 
and $\mu(x)$, respectively. Moreover, we make use of the notations $y = (z, \mu)$ and $y(x) =(z(x),\mu(x))$. 
In the subsequent discussions, $\sigma_{z_1}(\cdot), \sigma_{z_2}(\cdot), \cL_{z_1}(\cdot), \cL_{z_2}(\cdot)$ represent 
the partial derivatives of $\sigma(\cdot)$, $\cL(\cdot)$
with respect to $z_1$, $z_2$. For $u, w \in \{ x , y \}$, the notations 
$\nabla_u \cL(\cdot)$ and $\nabla^2_{uw} \cL(\cdot)$ correspond to the gradient of $\cL(\cdot)$ with respect to $u$ and
Hessian of $\cL(\cdot)$ for which the partial derivatives are first taken with respect to $u$ then 
with respect to $w$. 
\begin{definition}\label{assump1}
We call $x \in \Omega$ a nondegenerate point if
\begin{itemize}
\item[\bf (i)] the optimal $z$ of (\ref{eq:opt1}), denoted as $z(x)$, is unique,
\item[\bf (ii)] the singular value $\sigma(x,z(x))$ of $A(x,z(x))$ 
is simple, and
\item[\bf (iii)]$\nabla^2_{yy}\cL(x,y(x))$
is nonsingular.
\end{itemize} 
\end{definition}
For part \textbf{(iii)} of the definition above, we remark that its parts \textbf{(i)-(ii)} ensure
the uniqueness of $\mu(x)$, and so the uniqueness of $y(x)$ (i.e., see the proof of Theorem \ref{eq:psa_derivs}
in Appendix \ref{app:psa_der_proof}).

The next theorem below can be regarded as a generalization of the part of \cite[Theorem 8.1]{Burke2003}
that concerns the differentiability and first derivative of the map $X \in {\mathbb C}^{n\times n} \mapsto \alpha_{\epsilon}(X)$. 
Its proof exploits the optimality conditions for the constrained optimization
problem (\ref{eq:opt1}), as well as the implicit function theorem, and is given in Appendix \ref{app:psa_der_proof}.


\begin{theorem}[Derivatives of the Pseudospectral Abscissa Function]\label{eq:psa_derivs}
Let $\widetilde{x} \in \Omega$ be a nondegenerate point. Then the function $x \mapsto \alpha_{\epsilon}(x)$ is real-analytic
at $x = \widetilde{x}$ with
	\begin{equation}\label{eq:formula_1der}
				\nabla \alpha_\epsilon(\widetilde{x})	\;	=	\; 
				\left[	
					\begin{array}{ccc}
										\Real\left(\frac{ u^*  \left[  \frac{\partial A}{\:\, \partial x_1} (\widetilde{x}) \right]  v}{u^*v} \right)
											&		\cdots		&
										\Real\left(\frac{ u^* \left[ \frac{\partial A}{\:\, \partial x_d} (\widetilde{x}) \right]  v}{u^*v} \right)
					\end{array}
				\right]^T	\;
	\end{equation}
where $u$, $v$ consist of a pair of consistent unit left, right singular vectors, respectively, 
of $A(\widetilde{x}, z(\widetilde{x}))$ corresponding to $\sigma(\widetilde{x}, z(\widetilde{x}))$. Moreover,
\begin{equation}\label{eq:psa_2der}
	\nabla^2\alpha_\epsilon(\widetilde{x}) = \nabla^2_{xx}\cL(\widetilde{x},y(\widetilde{x})) - \nabla^2_{xy}\cL(\widetilde{x},y(\widetilde{x}))
				\cdot [\nabla^2_{yy}\cL(\widetilde{x},y(\widetilde{x}))]^{-1}[\nabla^2_{xy}\cL(\widetilde{x},y(\widetilde{x}))]^T
\end{equation}
where
\begin{equation*}
\begin{split}
	&	[\nabla^2_{xx}\cL(\widetilde{x},y(\widetilde{x}))]_{j\ell}  \; =  \;  
		-\mu(\widetilde{x}) \: 
		\left\{ \frac{\partial^2}{\partial x_j\partial x_\ell} [ \sigma(x,z) ] \bigg\vert_{\widetilde{x},z(\widetilde{x})} \right\}  \quad \quad j,\ell = 1,2,\cdots, d,		\\
	&	[\nabla^2_{xy}\cL(\widetilde{x},y(\widetilde{x}))]_{j\ell} \;  =  \; -\mu(\widetilde{x}) \: \left\{ \frac{\partial^2}{\partial x_j\partial y_\ell} [ \sigma(x,z) ] \bigg\vert_{\widetilde{x},z(\widetilde{x})} \right\} \quad \quad j = 1,2,\cdots d, \; \; \ell = 1,2,3,	\\
	&	[\nabla^2_{yy}\cL(\widetilde{x},y(\widetilde{x}))]_{j\ell}  \; =  \;  -\frac{\partial^2 }{\partial y_j\partial y_\ell} [ \mu \cdot \sigma(x,z) ] \bigg\vert_{\widetilde{x}, z(\widetilde{x}), \mu(\widetilde{x})} \quad \,    j,\ell = 1,2,3,
\end{split}
\end{equation*}
$(y_1,y_2,y_3) = (z_1,z_2,\mu)$, and $\, \mu(\widetilde{x}) = -1/ (u^\ast v)$. 
\end{theorem}

If we repeat the arguments above for the reduced pseudospectral abscissa function $\alpha_\epsilon^\cV(\cdot)$
for a given subspace $\cV$,  assuming the corresponding reduced $\epsilon$-pseudospectrum
is not empty in a neighborhood of the point under consideration,
analogous formulas for the derivatives of $\alpha_\epsilon^\cV(\cdot)$ can be 
obtained in terms of a matrix $V$ whose columns form an orthonormal basis for $\cV$. 
Now  $\alpha_\epsilon^\cV(\cdot)$ can be expressed as the optimization problem
\begin{equation}\label{eq:opt_red}
	\alpha_{\epsilon}^\cV (x) 		\;	=	\;	
			\max \{ z_1	\;\;		|	\;\; 	 z =  (z_1,z_2)\in\R^2, \;\;
								 \sigma^\cV (x, z )- \epsilon \leq 0	\}	,						
\end{equation}
and the Lagrangian for the optimization problem associated with $\alpha_\epsilon^\cV(\cdot)$
takes the form
\begin{equation}\label{eq:Lagrangian_red}
	\cL^\cV(x,z,\mu)  \;  :=  \;  z_1-\mu(\sigma^\cV(x,z)-\epsilon)	, 
\end{equation}
where $\sigma^\cV(x,z):= \sigma_{\min}(A^V(x,z))$ and $A^V(x,z):= A^V(x)-(z_1 + \ri z_2)V$.
We denote the optimal $z$ and the corresponding Lagrange multiplier $\mu$ 
for the optimization problem in (\ref{eq:opt_red}), assuming their uniqueness, 
by $z^\cV(x)$ and $\mu^\cV(x)$, and also let $y^\cV(x) = (z^\cV(x) , \mu^\cV(x))$ 
at a given $x \in \Omega$.
\begin{definition}\label{assump2}
Let ${\mathcal V}$ be a subspace of $\, {\mathbb C}^n$. 
We call $x \in \Omega$ a nondegenerate point of the restriction of $A(\cdot)$ 
to the subspace ${\mathcal V}$ if
\begin{itemize}
\item[\bf (i)] there is $\nu \in {\mathbb R}, \nu > 0$ such that
$\Lambda^{\mathcal V}_{\epsilon}(\widehat{x}) \neq \emptyset$ for all $\, \widehat{x} \in {\mathcal B}(x, \nu) \cap \Omega$, 
\item[\bf (ii)] the optimal $z$ of (\ref{eq:opt_red}), denoted as $z^\cV (x)$, is unique,
\item[\bf (iii)] the singular value $\sigma^\cV (x,z^\cV (x))$ of $A^V (x,z^\cV (x))$ 
is simple for any matrix $V$ whose columns form an orthonormal basis for $\cV$, and
\item[\bf (iv)]$\nabla^2_{yy}\cL^\cV (x,y^\cV (x))$
is nonsingular.
\end{itemize} 
\end{definition}

Applying steps analogous to the ones in the proof of Theorem \ref{eq:psa_derivs} yield
the following result.
\begin{theorem}[Derivatives of the Reduced Pseudospectral Abscissa Function]\label{eq:psa_derivs_red}
Let $\cV$ be a subspace of $\, {\mathbb C}^n$, 
and $\widetilde{x} \in \Omega$ be a nondegenerate point of the restriction of $A(\cdot)$ to $\cV$. 
Then the function $x \mapsto \alpha^\cV_{\epsilon}(x)$ is real-analytic at $x = \widetilde{x}$ with
	\begin{equation}\label{eq:formula_1der_red}
				\nabla \alpha^{\cV}_\epsilon(\widetilde{x})	\;	=	\; 
				\left[	
					\begin{array}{ccc}
			\Real\left(\frac{ (u^V)^\ast  \left[  \frac{\partial A^V}{\:\, \partial x_1} (\widetilde{x}) \right]  v^V }{(u^V)^\ast V v^V} \right)
											&		\cdots		&
			\Real\left(\frac{ (u^V)^\ast \left[ \frac{\partial A^V}{\:\, \partial x_d} (\widetilde{x}) \right]  v^V }{(u^V)^* V v^V } \right)
					\end{array}
				\right]^T	\;
	\end{equation}
where $V$ is a matrix whose columns form an orthonormal basis for $\cV$, and
$u^V, v^V$ form a pair of consistent unit left, unit right singular vectors of
$A^V(\widetilde{x},z^\cV(\widetilde{x}))$
corresponding to $\sigma^\cV(\widetilde{x},z^\cV(\widetilde{x}))$.
Furthermore,
\begin{equation}\label{eq:secondder2}
\begin{split}
\nabla^2\alpha_\epsilon^\cV(\widetilde{x}) 	\;	=	\;	 \nabla^2_{xx}\cL^\cV(\widetilde{x},y^\cV(\widetilde{x})) \;\,	- 	\hskip 44ex		\\
	\hskip 5ex		
	\nabla^2_{xy}\cL^\cV(\widetilde{x},y^\cV(\widetilde{x})) [\nabla^2_{yy}\cL^\cV(\widetilde{x},y^\cV(\widetilde{x}))]^{-1}[\nabla^2_{xy}\cL^\cV(\widetilde{x},y^\cV(\widetilde{x}))]^T ,
\end{split}
\end{equation}
where
\begin{equation*}
\begin{split}
	&	[\nabla^2_{xx}\cL^\cV(\widetilde{x},y^\cV(\widetilde{x}))]_{j\ell} \; = \;  
	-\mu^\cV (\widetilde{x}) \: 
		\left\{ \frac{\partial^2}{\partial x_j\partial x_\ell} [ \sigma^\cV (x,z) ] \bigg\vert_{\widetilde{x},z^\cV(\widetilde{x})} \right\}  \quad  j,\ell = 1,2,\cdots, d,	\\
	&	[\nabla^2_{xy}\cL^\cV(\widetilde{x},y^\cV(\widetilde{x}))]_{j\ell} \; = \; -\mu^\cV (\widetilde{x}) 
	\: \left\{ \frac{\partial^2}{\partial x_j\partial y_\ell} [ \sigma^\cV (x,z) ] \bigg\vert_{\widetilde{x},z^\cV (\widetilde{x})} \right\} \quad  \,  j = 1,2,\cdots d, \; \; \ell = 1,2,3, \\
	&	[\nabla^2_{yy}\cL^\cV(\widetilde{x},y^\cV(\widetilde{x}))]_{j\ell}  \; =  \;  -\frac{\partial^2 }{\partial y_j\partial y_\ell} 
	[ \mu \cdot \sigma^\cV (x,z) ] \bigg\vert_{\widetilde{x}, z^\cV (\widetilde{x}), \mu^\cV (\widetilde{x})}  \quad \;\;\:    j,\ell = 1,2,3,
\end{split}
\end{equation*}
and $\, \mu^\cV (\widetilde{x}) = -1/  \big( (u^V)^\ast \, V \, v^V \big)$.
\end{theorem}

\medskip

\medskip

\noindent
\textbf{Remark.} Expressions for the third and higher-order derivatives of $x \mapsto \alpha_\epsilon (x)$ 
at a nondegenerate point $\widetilde{x} \in \Omega$ can be derived by differentiating (\ref{eq:secondder}) further. 
Similarly, the higher-order derivatives of $x \mapsto \alpha^\cV_\epsilon (x)$ 
at a nondegenerate point of the restriction of $A(\cdot)$ to $\cV$ can be obtained
by differentiating the Lagrangian in (\ref{eq:Lagrangian_red}) repeatedly.

\subsection{Hermite Interpolation of the Pseudospectral Abscissa}\label{sec:Hermit_interpolate}
The next result concerns the Hermite interpolation properties between the full pseudospectral abscissa
function $\alpha_\epsilon(\cdot)$ and its reduced counterpart 
$\alpha^{\cV_k}_\epsilon(\cdot)$ for a subspace $\cV_k$ generated by Algorithm \ref{alg1}.
\begin{lemma}[Hermite Interpolation]\label{thm:hermite}
	The following assertions hold regarding the subspace $\cV_k$ for every integer $k \geq 2$
	and the points $x^{(\ell)}$, $z^{(\ell)}$ for $\ell = 2, \dots , k$ generated by Algorithm \ref{alg1}:
	\begin{itemize}
		\item[\bf{(i)}] $\Lambda^{{\mathcal V}_k}_{\epsilon}(x^{(\ell)}) \neq \emptyset \,$ and
		$\, \alpha_{\epsilon}(x^{(\ell)})=\alpha_{\epsilon}^{\mathcal{V}_k}(x^{(\ell)})$. 
		
		\item[\bf (ii)] The point $z^{(\ell)}$ is a rightmost point of $\Lambda^{\cV_k}_\epsilon(x^{(\ell)})$.
		
		\item[\bf{(iii)}] If $x^{(\ell)}$ is a nondegenerate point, then 		
		$\alpha_{\epsilon}(x)$ and $\alpha_{\epsilon}^{\mathcal{V}_k}(x)$  are differentiable 
		at $x=x^{(\ell)}$, and 
		$ \;
			\nabla\alpha_{\epsilon}(x^{(\ell)})=\nabla\alpha_{\epsilon}^{\mathcal{V}_k}(x^{(\ell)}). 
		$
	\end{itemize}
\end{lemma}

\begin{proof} 
\begin{itemize}
\item[\bf (i)]
From lines \ref{line:svk}-\ref{siter_end} of Algorithm \ref{alg1}, we have $v^{(\ell)} \in {\mathcal V}_k$, where $v^{(\ell)}$ is
a right singular vector corresponding to $\sigma_{\min}(A(x^{(\ell)})-z^{(\ell)}I)$ for the rightmost point $z^{(\ell)}$ in $\Lambda_\epsilon(x^{({\ell})})$,
for $\ell = 2, \dots k$. 
Thus, the assertion follows from Lemma \ref{lemma:interpol}.

\medskip

\item[\bf (ii)]
From Theorem \ref{thm:notempty_psa},
as $v^{(\ell)} \in {\mathcal V}_k$, we have $z^{(\ell)} \in \Lambda^{\cV_k}_\epsilon (x^{(\ell)})$, and so $\alpha_{\epsilon}(x^{(\ell)})  =  \Real(z^{(\ell)})  \leq   \alpha^{\cV_k}_{\epsilon}(x^{(\ell)})$.
But the reverse inequality   $\alpha_{\epsilon}(x^{(\ell)})  \geq   \alpha^{\cV_k}_{\epsilon}(x^{(\ell)})$  also holds due to Lemma \ref{monotonicity},
implying $\Real(z^{(\ell)})  =  \alpha^{\cV_k}_{\epsilon}(x^{(\ell)})$, i.e., $z^{(\ell)}$ is a rightmost point in $\Lambda^{\cV_k}_\epsilon(x^{(\ell)})$.

\medskip

\item[\bf (iii)]	
Differentiability of $\alpha_\epsilon (x)$ at $x = x^{(\ell)}$ is immediate, i.e., $x^{(\ell)}$ is nondegenerate,
so, from Definition \ref{assump1}, the rightmost point $z^{(\ell)}$ in $\Lambda_{\epsilon}(x^{(\ell)})$
is unique, and the smallest singular value of $\, A(x^{(\ell)}) - z^{(\ell)} I \,$ is simple. 

\medskip

\noindent
Let us show that $\alpha^{\cV_k}_{\epsilon}(x)$ is also differentiable at $x = x^{(\ell)}$. 
From part \textbf{(ii)}, $z^{(\ell)}$ is a rightmost point in $\Lambda^{\cV_k}_\epsilon(x^{(\ell)})$.
We claim that $z^{(\ell)}$ is the unique rightmost point in $\Lambda^{\cV_k}_\epsilon(x^{(\ell)})$. Suppose otherwise for the sake of contradiction,
that is there exists $\widetilde{z}\neq z^{(\ell)}$ such that  $\widetilde{z} \in \Lambda^{\cV_k}_\epsilon (x^{(\ell)})$ and
$\alpha_{\epsilon}^{\mathcal{V}_k}(x^{(\ell)}) = \Real(z^{(\ell)}) = \Real(\widetilde{z})$. 
But then, by Lemma \ref{monotonicity}, we have $\epsilon \geq \sigma_{\min}(A^{V_k}(x^{(\ell)})-\widetilde{z} \, V_k)\geq \sigma_{\min}(A(x^{(\ell)})-\widetilde{z}I)$, 
so $\widetilde{z} \in\Lambda_\epsilon(A(x^\ell))$. Consequently, $\widetilde{z}$ and $z^{(\ell)}$ are rightmost points 
in $\Lambda_\epsilon(x^{(\ell)})$, contradicting the assumption that $x^{(\ell)}$ is nondegenerate.
Additionally, following the arguments in the proof of Theorem \ref{thm:notempty_psa} 
in Appendix \ref{app:nonempty_psa_proof},
the simplicity of $\sigma_{\min}(A(x^{(\ell)})-z^{(\ell)}I)$ implies the simplicity
of $\sigma_{\min}(A^{{V}_k}(x^{(\ell)})-z^{(\ell)}V_k)$. Also, the 
right singular vector $v^{(\ell)}$ corresponding to $\sigma_{\min}(A(x^{(\ell)})-z^{(\ell)}I)$ is in ${\mathcal V}_k$,
so there is $\nu \in {\mathbb R}, \nu > 0$
such that $\Lambda^{{\mathcal V}_k}_{\epsilon}(x) \neq \emptyset$ for all $x \in {\mathcal B}(x^{(\ell)},\nu) \cap \Omega$
due to Theorem \ref{thm:notempty_psa}. 
Since $z^{(\ell)}$ is the unique rightmost point in $\Lambda^{\cV_k}_\epsilon(x^{(\ell)})$,
the singular value $\sigma_{\min}(A^{{V}_k}(x^{(\ell)})-z^{(\ell)}V_k)$ is simple, and 
$\Lambda^{{\mathcal V}_k}_{\epsilon}(x) \neq \emptyset$ for all $x$ in an open neighborhood of $x^{(\ell)}$,
we conclude that $\alpha_{\epsilon}^{\mathcal{V}_k}(x)$ is differentiable at $x=x^{(\ell)}$.
	
\medskip	
	
Finally, we show that the gradients of $\alpha_\epsilon(x)$ and $\alpha^{\cV_k}_\epsilon(x)$
are equal at $x=x^{(\ell)}$. Let $u,v$ be a consistent pair of unit left, unit right singular vectors corresponding to 
$\sigma_{\min}(A(x^{(\ell)})-z^{(\ell)}I)$. Since $v\in\mathcal{V}_k$, there is a unit vector $a$ such 
that $v=V_k a$. From the arguments in the proof of Theorem \ref{thm:notempty_psa}, the vectors $u,a$ 
must be a consistent pair of unit left, unit right singular vectors 
corresponding to $\sigma_{\min}(A^{V_k}(x^{(\ell)})-z^{(\ell)}V_k)$. 
Using the analytical formulas derived in the previous subsection for the derivatives of
$\alpha_\epsilon(\cdot)$ and $\alpha^{\cV_k}_\epsilon()$, specifically
using (\ref{eq:formula_1der})  and (\ref{eq:formula_1der_red}), we obtain
\begin{equation*}
\begin{split}
	\frac{\partial\alpha_{\epsilon}}{\partial x_j}( x^{(\ell)} )
		&	=		\Real\left(\frac{u^*  \left[ \frac{\partial \, A}{\partial x_j} (x^{(\ell)}) \right] v}{u^* v}\right) \\
	 	&	=	 	\Real\left(\frac{u^*  \left[ \frac{\partial \, A^{{V}_k}}{\partial x_j}(x^{(\ell)}) \right] a}{u^* V_k a}\right) 
							\;	=	\; 	\frac{\partial\alpha_{\epsilon}^{\mathcal{V}_k}}{\partial x_j}( x^{(\ell)} )
\end{split}							
\end{equation*}
for $j = 1, \dots, d$.
\end{itemize}
\end{proof}

\section{Convergence of the Subspace Framework}\label{sec:rate_conv}
This section concerns the convergence of the sequence $x^{(2)}, x^{(3)}, x^{(4)}, \dots $ generated 
by Algorithm \ref{alg1}. The first result below asserts that if any two of these iterates are equal, 
then global convergence is achieved.

\begin{theorem}\label{thm:gconv1}
If two points $x^{(\ell)}, x^{(k)}$ with $2 \leq \ell < k$ generated by Algorithm \ref{alg1}
are equal, then $x^{(\ell)}$ is a global minimizer of $\alpha_\epsilon(x)$ over all $x\in \underline{\Omega}$.
\end{theorem}
\begin{proof}
It follows from the monotonicity property (i.e., part \textbf{(iii)} of Lemma \ref{monotonicity}) that
\[
	\alpha^{\cV_{\ell}}_\epsilon (x^{(k)})	\;	\leq	\;	
	\alpha^{\cV_{k-1}}_\epsilon (x^{(k)})	\;	=	\;	\min_{x \in \underline{\Omega}}	\, \alpha^{\cV_{k-1}}_\epsilon(x)
		\;	\leq		\;
	\min_{x \in \underline{\Omega}}	\, \alpha_\epsilon(x).
\]
Moreover, from the interpolation property (i.e., part \textbf{(i)} of Lemma \ref{thm:hermite}), we have
\[
	\min_{x \in \underline{\Omega}}	\, \alpha_\epsilon(x) 
		\;	\leq	\;
	\alpha_\epsilon (x^{(\ell)})	 	\;	=	\;	\alpha^{\cV_{\ell}}_\epsilon (x^{(\ell)})
							\;	=	\;	\alpha^{\cV_\ell}_\epsilon (x^{(k)}).
\]
Combining the inequalities above,
$\min_{x \in \underline{\Omega}}	\, \alpha_\epsilon(x) \,	=	\,	\alpha^{\cV_\ell}_\epsilon (x^{(k)})	
				\,	=	\,		\alpha_\epsilon (x^{(\ell)})$
as desired.
\end{proof}
\smallskip
An important conclusion about the global convergence of Algorithm \ref{alg1} that can be drawn from
Theorem \ref{thm:gconv1} is as follows.
The subspace $\cV_{\ell}$ is a subspace of ${\mathbb C}^n$ and contains $\cV_{\ell-1}$ 
for every $\ell > 1$. As a result, we must have $\cV_{\ell} = \cV_{\ell-1}$ (so that $x^{(\ell+1)} = x^{(\ell)}$)
for some $\ell > 1$. Theorem \ref{thm:gconv1} implies $x^{(\ell)}$ for such an $\ell$ is a global 
minimizer of $\alpha_\epsilon(x)$ over $x \in \underline{\Omega}$. 

\smallskip

The rest of this section is devoted to proving that the rate of convergence of 
Algorithm \ref{alg1} when $d=1$ is superlinear under mild assumptions. 
We can generalize the arguments and the result when $d > 1$ provided that additional singular vectors 
are put into the subspace at every iteration at points close to the interpolation point employed by Algorithm \ref{alg1};
this extension is similar to the extension of Algorithm 1 to Algorithm 2 in \cite{Kangal2018} in
the context of minimizing the $j$th largest eigenvalue for a prescribed $j$. We first focus on
the case $d = 1$ for the sake of simplicity. In a subsection at the end of the section, we formally 
introduce its extension to attain superlinear convergence in theory when $d > 1$, and outline 
why the convergence results for Algorithm \ref{alg1} applies to this extension

Throughout, it is assumed that three consecutive iterates $ x^{(k-1)}, x^{(k)}, x^{(k+1)} $
generated by Algorithm \ref{alg1} for $d=1$ are sufficiently close to $x_*$,
a global minimizer of $\alpha_\epsilon(x)$ over $x \in \Omega$. 
An additional assumption that is kept throughout is that $x_\ast$ is strictly in the
interior of $\underline{\Omega}$. Moreover, by $z_\ast$
we denote a point in $\Lambda_{\epsilon}(x_\ast)$ satisfying $\alpha_\epsilon(x_\ast) = \Real(z_\ast)$,
whereas ${\mathcal R}(z_\ast) := (\Real(z_\ast) , \Imag(z_\ast)) \in {\mathbb R}^2$.
Furthermore, we assume the nondegeneracy of the point $x_\ast$, which
is stated formally below.

\begin{assumption}[Nondegeneracy of Optimizer]\label{ass:optimality}
The point $x_\ast$ is in the interior of $\underline{\Omega}$, and
is a nondegenerate point, that is the following conditions are satisfied by 
$x_\ast$, $z_\ast$:
\begin{itemize}
	\item [\bf (i)] $z_\ast$ is the unique point in $\Lambda_\epsilon(x_\ast)$
	such that  $\Real(z_\ast) = \alpha_\epsilon(x_\ast)$.
	\item [\bf (ii)] the smallest singular value $\sigma(x_\ast,{\mathcal R}(z_\ast))$ of $A(x_\ast, {\mathcal R}(z_\ast))$
	is simple.
	\item[\bf (iii)] $\nabla^2_{yy} {\mathcal L}(x_\ast , y_\ast)$ is invertible, where
	${\mathcal L}(x, y)$ is the Lagrangian as in (\ref{eq:Lagrange}) with $y = (z, \mu)$,
	and $y_\ast = ({\mathcal R}(z_\ast) , \mu_\ast)$ with $\mu_\ast \in {\mathbb R}$
	represents the optimal point satisfying $\nabla_{y}\cL(x_\ast, y_\ast) = 0$.
\end{itemize}
\end{assumption}

By Theorem \ref{eq:psa_derivs}, nondegeneracy assumption above for $x_\ast$
guarantees that the pseudospectral abscissa map $x \mapsto \alpha_{\epsilon}(x)$ is 
real analytic at $x = x_\ast$. 
We keep another assumption throughout this section,
which concerns the second derivatives of the Lagrangian and its reduced
counterpart. Note that, throughout this section, $y_\ast =  (z_\ast , \mu_\ast)$
with $\mu_\ast \in {\mathbb R}$ is as in Assumption \ref{ass:optimality}, that is it represents
the unique $y_\ast$ satisfying $\nabla_{y}\cL(x_\ast, y_\ast) = 0$.

\begin{assumption}[Robust Nondegeneracy]\label{ass:optimality2} 
For a given constant $\gamma>0$, the following assertions hold:
\[
	\sigma_{\min}\left(\nabla_{yy}^2\cL(x_\ast, y_\ast) \right) \;		\geq	\;	 \gamma 
		\quad\;\;	 \text{and} 	\quad\;\; 
	\sigma_{\min}\left(\nabla_{yy}^2\cL^{\cV_k}(x_\ast , y_\ast) \right) \geq \gamma.
\]
\end{assumption}

We start our derivation of the rate of convergence with a lemma that asserts that the pseudospectral 
abscissa functions of the full and the reduced problems are attained at a unique smooth optimizer (i.e., a unique 
right-most point in pseudospectra, that is smooth with respect to $x$) around $x_\ast$ 
under Assumptions \ref{ass:optimality} and \ref{ass:optimality2}. We omit its proof, as the proof of its part \textbf{(i)}
is similar to \cite[Proposition 2.9]{Kangal2018}, and proofs of parts \textbf{(ii)}-\textbf{(iii)} to \cite[Lemma 15]{Mengi2018}.
\begin{lemma}\label{lemma:uniqness}
Suppose that Assumption \ref{ass:optimality} and \ref{ass:optimality2} hold. There exist 
$\nu_x, \nu_z,\nu_\mu>0$ independent of $k$ such that $\mathcal{B}(x_\ast,\nu_x) \subseteq \underline{\Omega}$
that satisfy the following:

\begin{itemize}
\item[\bf (i)] The singular value functions $\sigma(x,z)$ and  $\sigma^{\cV_k}(x,z)$ are 
simple, and their first three derivatives in absolute value are bounded above by constants 
uniformly for all $x \in \mathcal B(x_\ast, \nu_x)$, $z \in \mathcal B({\mathcal R}(z_\ast), \nu_z)$,  
where the constants are independent of $k$.

\item[\bf (ii)] There exists a real-analytic function 
\[
	x \in \mathcal{B}(x_\ast , \nu_x) \: \mapsto \: 
		y(x) = (z(x),\mu(x)) \in \mathcal{B}({\mathcal R}(z_\ast) , \nu_z) \times \mathcal{B}(\mu_\ast, \nu_\mu)
\]
such that $\alpha_\epsilon(x)=z_1(x)$ with $z_1(x)$ denoting the first component of $z(x) \in {\mathbb R}^2$, as well as
\begin{equation}
	\nabla_y\cL(x,y(x)) \; = \; 0 	\quad\;\;	 \text{and} 	\quad\;\; 
			\sigma_{\min}(\nabla_{yy}^2\cL(x,y(x)) ) \; \geq \; \gamma/2
	\end{equation}
for all $x \in  \mathcal{B}(x_\ast,\nu_x)$.

\item[\bf (iii)] 
There exists a real-analytic function 
\[						
	x \in \mathcal{B}(x_\ast , \nu_x) \: \mapsto \: 
  y^{\cV_k}(x) = (z^{\cV_k}(x),\mu^{\cV_k}(x)) \in  \mathcal B({\mathcal R}(z_\ast), \nu_z) \times \mathcal B(\mu_\ast, \nu_\mu)
\] 
such that $\alpha_\epsilon^{\cV_k}(x)=z^{\cV_k}_1(x)$ with
$z^{\cV_k}_1(x)$ denoting the first component of $z^{\cV_k}(x) \in {\mathbb R}^2$, and
\begin{equation}
	\nabla_y\cL^{\cV_k}(x,y^{\cV_k}(x)) \;	 = \; 0 	\quad\;\;	 \text{and} 	\quad\;\; 
		\sigma_{\min}(\nabla_{yy}^2\cL^{\cV_k}(x,y^{\cV_k}(x)) ) \; \geq \; \gamma/2
\end{equation}
for all $x \in  \mathcal{B}(x_\ast,\nu_x)$.

\end{itemize}
\end{lemma}

For the main rate-of-convergence result, as stated formally below, we also assume that the angle 
between the left and right singular vectors of $\sigma^{\cV_k}(x, z^{\cV_k}(x))$ obeys a certain 
bound on ${\mathcal B}(x_\ast, \nu_x)$, the interval in Lemma \ref{lemma:uniqness}. 
Specifically, letting $u(x)$ and $v(x)$ be a pair of consistent unit left and right singular vectors corresponding to 
$\sigma(x, z(x))$, following the arguments in Section \ref{sec:psa_derivatives}, in particular by an application of 
the first order optimality conditions to the constrained optimization characterization in (\ref{eq:opt1}) of $\alpha_{\epsilon}(x)$,
we have $u(x)^\ast v(x) < 0$ for all $x \in {\mathcal B}(x_\ast, \nu_x)$ so that
\begin{equation}\label{eq:bound_svecs}
	\underline{\beta}	:=		\max_{x \in {\mathcal B}(x_\ast, \nu_x)} u(x)^\ast v(x)	\;	<	\;	0.
\end{equation}
Similarly, it can be shown that $u^{V_k}(x)^\ast V_k v^{V_k}(x) < 0$ for 
all $x \in {\mathcal B}(x_\ast, \nu_x)$, where $u^{V_k}(x)$ and $v^{V_k}(x)$ denote a consistent
pair of unit left and right singular vectors of $A^{V_k}(x, z^{\cV_k}(x))$ corresponding to 
$\sigma^{\cV_k}(x, z^{\cV_k}(x))$.
Indeed, let $\widetilde{x} \in {\mathcal B}(x_\ast, \nu_x)$.
Exploiting $\nabla_y\cL^{\cV_k}(\widetilde{x},y^{\cV_k}(\widetilde{x})) \;	 = \; 0$ with
$\cL^{\cV_k}(x,z,\mu)  \;  :=  \;  z_1-\mu(\sigma^{\cV_k}(x,z)-\epsilon)$, we deduce
 \begin{equation*}
 	\begin{split}
0	\;	&	=	\;
1 - \mu^{\cV_k}(\widetilde{x}) \Real\left(  u^{V_k}(\widetilde{x})^\ast 
\frac{\partial}{\partial z_1}  \left[ A^{V_k}(x,z) \right] \bigg|_{\widetilde{x} , z^{\cV_k}(\widetilde{x}) } v^{V_k}(\widetilde{x})  \right) \\
	\;	&	=	\;	 1+\mu^{\cV_k} (\widetilde{x}) \Real\left(  u^{V_k}(\widetilde{x})^\ast V_k v^{V_k}(\widetilde{x}) \right)
	\end{split}
\end{equation*}
and 
\begin{equation*}
\begin{split}
0	\;	&	=	\;	 - \mu^{\cV_k}(\widetilde{x}) \Real\left(  u^{V_k}(\widetilde{x})^\ast , \frac{\partial}{\partial z_2}  \left[ A^{V_k}(x,z) \right] \bigg|_{\widetilde{x} , z^{\cV_k}(\widetilde{x}) } v^{V_k}(\widetilde{x})  \right) 	\\
		&	=	\;	 -\mu^{\cV_k}(\widetilde{x}) \Imag(  u^{V_k}(\widetilde{x})^\ast V_k v^{V_k}(\widetilde{x})  ).
\end{split}
\end{equation*}
From the first equation, we have $\mu^{\cV_k} (\widetilde{x}) \neq 0$. As a result, from the second equation, 
$\Imag(  u^{V_k}(\widetilde{x})^\ast V_k v^{V_k}(\widetilde{x})  ) = 0$,
which shows $ u^{V_k}(\widetilde{x})^\ast V_k v^{V_k}(\widetilde{x}) $ is real. 
Additionally, $ u^{V_k}(\widetilde{x})^\ast V_k v^{V_k}(\widetilde{x}) \neq 0$ due to the first equation.
Moreover, from the complementary conditions as $\mu^{\cV_k} (\widetilde{x}) \neq 0$,
we have $\sigma^{\cV_k}(\widetilde{x},z^{{\mathcal V}_k}(\widetilde{x})) = \epsilon$, so
$z^{{\mathcal V}_k}(\widetilde{x})$ is the rightmost point in $\Lambda^{{\mathcal V}_k}_\epsilon(\widetilde{x})$, 
which in turn implies
\[
	0  \; \leq  \; 
	\frac{ \partial}{\partial z_1} 
			\left[ \sigma_{\min} \left( A^{V_k}(x,z) \right) \right] \bigg|_{\widetilde{x}, z^{\cV_k}(\widetilde{x})}
							\;	=	\;	 
		 -u^{V_k}(\widetilde{x})^\ast V_k v^{V_k}(\widetilde{x}) .
\]
Combining this with $u^{V_k}(\widetilde{x})^\ast V_k v^{V_k}(\widetilde{x})  \neq 0$ yields 
$0 >  u^{V_k}(\widetilde{x})^\ast V_k v^{V_k}(\widetilde{x})$ as desired.

Below, we require slightly more, i.e., there is a negative upper bound independent of $k$
on $u^{V_k}(x)^\ast V_k v^{V_k}(x)$ over all $x \in {\mathcal B}(x_\ast, \nu_x)$.

\begin{assumption}\label{ass:optimality3}
For a given constant $\beta$ small enough in absolute value and such that $\underline{\beta} \leq \beta < 0$, 
the subspace $\cV_k$ is such that
\[
	\beta 	\;  \geq  \;		\max_{x \in {\mathcal B}(x_\ast, \nu_x)} u^{V_k}(x)^\ast V_k v^{V_k}(x),
\]
where ${\mathcal B}(x_\ast, \nu_x)$ is as in Lemma \ref{lemma:uniqness}.
\end{assumption}


The next result concerns the uniform boundedness of the derivatives of 
$\alpha_\epsilon(x)$ and $\alpha^{\cV_k}_\epsilon(x)$.

\begin{lemma}\label{lemma:diff}
Suppose that the conditions of Assumptions \ref{ass:optimality}, \ref{ass:optimality2} and \ref{ass:optimality3} hold. 
There exists $\nu_x$ independent of $k$ such that ${\mathcal B}(x_\ast, \nu_x) \subseteq \underline{\Omega}$,
and the following assertions hold:
\begin{enumerate}
	\item[\bf (i)] The pseudospectral functions $\alpha_\epsilon(x)$, 
	$\alpha_\epsilon^{\cV_k}(x)$ are real analytic at all $x\in \mathcal B(x_\ast, \nu_x)$. 
	\item[\bf (ii)] The first three derivatives of $\alpha_\epsilon(x)$ and $\alpha_\epsilon^{\cV_k}(x)$
	in absolute value are bounded above by constants uniformly for all $x\in \mathcal B(x_\ast, \nu_x)$, 
	where the constants are independent of $k$.
\end{enumerate}
\end{lemma}

\begin{proof}
	The real-analyticity of $\alpha_\epsilon(x)$ and $\alpha_\epsilon^{\cV_k}(x)$
	uniformly in an interval $\mathcal B(x_\ast, \nu_x)$ for some $\nu_x$ that is independent of $k$
	is immediate from Lemma \ref{lemma:uniqness}.

	For the proof of boundedness of the first derivative of $\alpha_\epsilon(x)$,  
	noting $u(x)^*v(x) \leq \beta < 0$ for $x \in \mathcal{B}(x_\ast, \nu_x)$
	due to (\ref{eq:bound_svecs}), and using
	the formulas (\ref{eq:formula_1der}) for the derivatives of $x \mapsto \alpha_{\epsilon}(x)$,
	at any $\widetilde{x} \in \mathcal{B}(x_\ast, \nu_x)$, we have
	\[
		 \alpha'_{\epsilon}(\widetilde{x})
					\;	=	\;
			\text{Re}\left(
				\sum_{\ell=1}^{\kappa}\frac{\mathrm{d} f_\ell({\widetilde{x}})}{\mathrm{d} x}
					\frac{u(\widetilde{x})^\ast A_\ell v(\widetilde{x})}{u(\widetilde{x})^\ast v(\widetilde{x})}
			\right)	
					\;\;	\Longrightarrow		\;\;\;
			\abs{ \alpha'_{\epsilon}({\widetilde{x}}) }
			\;	\leq		\;
			\sum_{\ell=1}^{\kappa}\abs{\frac{\mathrm{d} f_\ell({\widetilde{x}})}{\mathrm{d} x}} \frac{ \| A_\ell \|_2}{ | \beta |} ,
	\]
	implying  the boundedness of $\abs{\alpha'_{\epsilon}(x)}$ on $\mathcal B(x_\ast, \nu_x)$, 
	as $f_\ell$ is real analytic for $ \ell = 1, \dots, \kappa$. As for the second derivatives of $\alpha_\epsilon(x)$, 
	first observe
	\begin{equation}\label{eq:mux}
		\mu(x) = -\frac{1}{u(x)^*v(x)} \; \leq \; \frac{1}{-\beta}
	\end{equation}
	for all $x\in \mathcal B(x_\ast,\nu_x)$.
	Moreover, from part \textbf{(ii)} of Lemma \ref{lemma:uniqness}, we have
	\begin{equation}\label{eq:sigminx}
		\sigma_{\min}(\nabla_{yy}^2\cL(x,y(x)))  \; \geq \;  \gamma/2
	\end{equation}
	for all $x \in  \mathcal B(x_\ast,\nu_x)$,
	where $\gamma$ is as in Assumption \ref{ass:optimality2}.
	Formula (\ref{eq:psa_2der}) for the second derivative of $\alpha_\epsilon(x)$ yields
	\begin{equation}\label{eq:psa_sder_abs}
		\left|   \alpha''_\epsilon(x) \right|  \; \leq \; | \cL_{xx} (x,y(x))  |
			+\left\|\nabla_{xy}^2\cL(x,y(x))\right\|_2^2 \Vert [\nabla_{yy}^2\cL(x,y(x))]^{-1} \Vert_2,
	\end{equation}
	where $\cL_{xx} (x,y)$ is the second derivative of $\cL(x,y)$ with respect to $x$.
	Now, since the derivatives  $\cL_{xx} (x,y(x))$
	and $\nabla_{xy}^2\cL(x,y(x))$ can be expressed fully in terms of $\mu(x)$ and the partial derivatives 
	of $\sigma(x,y(x)),$ the boundedness of $\alpha''_{\epsilon}(x)$
	 follows from (\ref{eq:psa_sder_abs}) by employing \eqref{eq:mux}, \eqref{eq:sigminx} and part \textbf{(i)} of
	Lemma \ref{lemma:uniqness}. 
	
	The uniform boundedness of the third derivatives of $\alpha_{\epsilon}(x)$, and the boundedness of 
	the derivatives of the reduced pseduospectral abscissa function can be shown in a similar way.
\end{proof}

Interpolation properties between $\alpha_\epsilon (x)$ and $\alpha^{\cV_k}_\epsilon (x)$ and their
first derivatives hold at $x^{(k)}$. Even if these interpolation properties do not extend to the second derivatives, 
the second derivatives must be close at $x^{(k)}$ as shown next.

\begin{lemma}[Proximity of the Second Derivatives]\label{eq:prox_2der}
Suppose that Assumptions \ref{ass:optimality}, \ref{ass:optimality2} and \ref{ass:optimality3} are satisfied.
There exists a constant $C > 0$ independent of $k$ such that
		\begin{equation}\label{lm:slc0}
		\big|	\alpha''_{\epsilon} (x^{(k)}) - 
					[ \alpha_{\epsilon}^{\cV_k} ]'' (x^{(k)})	  \big|
							\:	\leq	\:	 C	 \big| x^{(k)}-x^{(k-1)}  \big| \: .
		\end{equation}
\end{lemma}

\begin{proof}
		We assume without loss of generality that $x^{(k-1)}$, $x^{(k)}$ are strictly inside ${\mathcal B}(x_\ast, \nu_x)$
		(i.e., the interval in Lemma \ref{lemma:diff}), where $\alpha_\epsilon(x)$ and $\alpha_\epsilon^{\cV_k}(x)$ 
		are real analytic with uniform bounds on their first three derivatives independent of $k$.
		
		Setting $\, h := x^{(k-1)} - x^{(k)} \,$, let us introduce the functions
		\begin{equation}\label{lm:slc1}
			l (t) := \alpha_{\epsilon} (x^{(k)} + t h) \text{ and } l_{k} (t) := \alpha_{\epsilon}^{\cV_{k}}(x^{(k)} + t h)
		\end{equation}
		for $t \in [0,1]$. By applying Taylor's theorem with third order remainder to $l (t)$ and $ l_k(t)$ 
		on the interval $(0,1)$, we obtain
	    	\begin{align*}
			& 	   l (1) 		 \;\;  = \;  		l (0) + l' (0) + l'' (0) /2 + l''' (\eta) /6,\\
			&	    l_{k} (1) 	 \;  = \;	  	l_{k} (0) + l'_{k} (0) + l''_{k} (0) /2 + l'''_{k} (\eta_k)/6
		\end{align*}
	    	for some constants $\eta, \eta_k \in (0,1)$. Notice that $l (1)=l_{k} (1)$,  $l (0)=l_{k} (0)$, and  
		$l'(0)=l'_{k} (0)$ due to Lemma \ref{thm:hermite}. Consequently,
	    	\begin{equation}\label{eq:inter}
			\begin{split}
			\frac{ \alpha''_{\epsilon} (x^{(k)}) h^2  -  [\alpha^{\cV_k}]'' (x^{(k)}) h^2 }{2}
								&	\;	=	\;	\frac{l''(0) - l''_{k} (0)}{2}	\\
								&	\;	=	\;	\frac{l'''_{k} (\eta_k)-l'''(\eta)}{6}	\\
								&	\;	=	\;
			\frac{ [\alpha^{\cV_k}]''' (x^{(k)} + \eta_k h) h^3	-	\alpha'''_{\epsilon} (x^{(k)} + \eta h) h^3   }{6}
				\,	.
			\end{split}
		\end{equation}
	    As $x^{(k)} + \eta_k h, x^{(k)} + \eta h \in {\mathcal B}(x_\ast, \nu_x)$, the third derivatives 
	   $[\alpha^{\cV_k}]''' (x^{(k)} + \eta_k h)$ and $\alpha'''_{\epsilon} (x^{(k)} + \eta h)$ on the righthand
	   side of (\ref{eq:inter}) in absolute value are bounded from above by a uniform constant $U$. 
	   Hence, we deduce from (\ref{eq:inter}) that
	   \[
	   		\big|	\alpha^{''}_{\epsilon} (x^{(k)}) - 
					[ \alpha_{\epsilon}^{\cV_k} ]'' (x^{(k)})	  \big|
							\:	\leq	\:	 \frac{2U}{3}  h	\;	=	\;	\frac{2U}{3}	 \big| x^{(k)}-x^{(k-1)}  \big| \:
	   \]
	   as desired.
\end{proof}

Now we are ready to state and prove the main result of this section, that is the superlinear
convergence result regarding the iterates of Algorithm \ref{alg1}.

\begin{theorem}[Superlinear Convergence]\label{thm:super_converge}
	Suppose that Assumptions \ref{ass:optimality}, \ref{ass:optimality2} and \ref{ass:optimality3} hold. 
	Additionally, assume that
	$\alpha''_\epsilon (x_\ast) \neq 0$.
	Then, there exists a constant $\Upsilon > 0$ independent of $k$ such that
	\begin{equation}\label{slc0}
	\frac{ \big| x^{(k+1)} - x_*  \big| }{ \big| x^{(k)} - x_*  \big| \max \{ \big| x^{(k)} - x_*  \big| ,  \big| x^{(k-1)} - x_*  \big| \}} 	\;	\leq	\;	 \Upsilon. 
	\end{equation}
\end{theorem}

\begin{proof}
	Let $\nu_x$ be as in Lemma \ref{lemma:diff} so that $\alpha_\epsilon(x)$ and $\alpha_\epsilon^{\cV_k}(x)$ are differentiable,
	indeed $\alpha''_\epsilon(x)$ and $[\alpha_\epsilon^{\cV_k}]''(x)$ are Lipschitz continuous, inside ${\mathcal B}(x_\ast, \nu_x)$.
	Moreover, without loss of generality, assume $x^{(k-1)}, x^{(k)}, x^{(k+1)} \in {\mathcal B}(x_\ast, \nu_x)$.

	By assumption $\alpha''_\epsilon (x_\ast) \neq 0$, so, by employing Lemma \ref{eq:prox_2der},
	we can assume $x^{(k-1)}$, $x^{(k)}$ are close enough so that $[\alpha_\epsilon^{\cV_k}]''(x_\ast) \neq 0$.
	Indeed, if necessary by choosing $\nu_x$ even smaller, we can assume $\alpha''_\epsilon(x) \neq 0, \, [\alpha_\epsilon^{\cV_k}]''(x) \neq 0$
	for all $x \in {\mathcal B}(x_\ast, \nu_x)$.

	The superlinear convergence assertion in (\ref{slc0}) now follows by expanding $\alpha'_\epsilon(x_\ast)$
	about $x^{(k)}$ by employing Taylor's theorem with integral remainder. The derivation here is identical to that
	in Part 2 in the proof of \cite[Theorem 3.3]{Aliyev2017} by replacing $\sigma(\omega), \sigma_r(\omega)$
	in that work with $\alpha_\epsilon(x), \alpha_\epsilon^{\cV_k}(x)$, and $\omega_\ast, \omega_{r+1}, \omega_{r}, \omega_{r-1}$
	in that work with $x_\ast, x^{(k+1)}, x^{(k)}$, $x^{(k-1)}$.
\end{proof}

In the context of a related subspace framework for maximizing the smallest eigenvalue 
of a matrix-valued function dependent on one parameter, the order of convergence is shown 
to be at least $1 + \sqrt{2}$ in \cite{Kressner2018}. It is possible that the order of convergence
of Algorithm \ref{alg1} is also better than quadratic, which is suggested by the numerical
experiments we have performed.

\subsection{Extended Subspace Framework}
When $d > 1$ the arguments above leading to the superlinear convergence assertion
for Algorithm \ref{alg1} fails, because the generalization of (\ref{lm:slc0})
that concerns the proximity of the second partial derivatives of $\alpha_{\epsilon}(x)$,
$\alpha^{\cV_k}_\epsilon(x)$ at $x^{(k)}$ does not have to hold.

A remedy to this issue is proposed in Algorithm \ref{alg2}. 
The difference of Algorithm \ref{alg2} compared to Algorithm \ref{alg1}
is that additional right singular vectors at points close to $x^{(k)}$ are also included
when forming the subspace $\cV_k$.
Let $h^{(k)} = \| x^{(k)} - x^{(k-1)} \|$, and $e_{pq} = 1/\sqrt{2} (e_p + e_q)$ if 
$p \neq q$ and $e_{pp} = e_p$, where $e_p$ denotes the $p$th column of the 
$d\times d$ identity matrix. 
To be precise, in lines \ref{a2:in_for}-\ref{a2:in_for_end} of Algorithm \ref{alg2}, a right 
singular vector $v^{(k)}_{pq}$ corresponding to $\sigma_{\min}(A(x^{(k)}_{pq})-z^{(k)}_{pq}I)$ is 
computed at $x^{(k)}_{pq} = x^{(k)} + h^{(k)} e_{pq}$ and $z^{(k)}_{pq}$, a rightmost
point in $\Lambda_{\epsilon}(x^{(k)}_{pq})$, for $p = 1, \dots d$, $q = p, \dots , d$.
In addition to $v^{(k)}$, all of these singular vectors are also included in the subspace $\cV_k$.
Clearly, this approach is practical when there are only a few parameters.
As for Algorithm \ref{alg1}, we assume $\Lambda_{\epsilon}^{{\mathcal V}_1}(x) \neq \emptyset$
for all $x \in \underline{\Omega}$. Under the simplicity assumption on the smallest singular value
at the rightmost points of the $\epsilon$-pseudospectra,
if the initial interpolation points $x^{(1)}_1 , \dots , x^{(1)}_\eta$ in line \ref{alg:init_points_2}
of Algorithm \ref{alg2} are chosen on a uniform grid for $\underline{\Omega}$
sufficiently fine, then an implication of
Theorem \ref{thm:notempty_psa} is that this assumption is met, and, as a result,
$\Lambda^{{\mathcal V}_k}_{\epsilon}(x) \neq \emptyset$
for all $x \in \underline{\Omega}$ for all $k \geq 1$ due to monotonicity.

All of the results and derivations when deducing the superlinear convergence for Algorithm \ref{alg1}
when $d = 1$ hold for Algorithm \ref{alg2} when $d > 1$ with minor modifications. To be specific,
the Hermite interpolation property, that is Lemma \ref{thm:hermite} holds, not only
at $x^{(\ell)}$ but also at nearby $x^{(\ell)}_{pq}$ for $p = 1, \dots , d$, $q = p, \dots , d$.
Moreover, Lemma \ref{lemma:uniqness} extends without modification except that ${\mathcal B}(x_\ast, \nu_x)$
is now not an interval but a ball in ${\mathbb R}^d$, and Lemma \ref{lemma:diff} 
holds after replacing all instances of the derivatives with partial derivatives. 
In the proof of Lemma \ref{eq:prox_2der}, we now assume without loss of generality that
not only $x^{(k)}, x^{(k-1)}$ but also $x^{(k)}_{pq}$ for $p = 1, \dots, d$ and $q = p, \dots , d$,
are contained in the ball ${\mathcal B}(x_\ast, \nu_x)$,
and the proximity result
\[
	\left|
		\frac{\partial^2 \alpha_{\epsilon}}{\partial x_p \partial x_q}(x^{(k)})
				-
		\frac{\partial^2 \alpha^{\cV_k}_{\epsilon}}{\partial x_p \partial x_q}(x^{(k)})
	\right|
		\;	\leq 	\;	C	\| x^{(k)} - x^{(k-1)} \|_2
\]
for a constant $C$ independent of $k$ is obtained by applying Taylor's theorem to
$l_{pq}(t) := \alpha_{\epsilon} (x^{(k)} + t h^{(k)} e_{pq})$ and 
$l_{k,pq} (t) := \alpha_{\epsilon}^{\cV_{k}}(x^{(k)} + t h^{(k)} e_{pq})$ for $p = 1, \dots , d$ and
$q = p, \dots , d$. Finally, the extension of Lemma \ref{thm:super_converge}, that is the main
superlinear convergence result, is stated formally below for Algorithm \ref{alg2} when $d > 1$. 
Its derivation is based on Taylor's theorem with integral remainder, and similar to the proof of 
\cite[Theorem 3.3]{Kangal2018}.
\begin{theorem}\label{thm:super_converge2}
	Let $x_\ast$ be a global minimizer of $\alpha_{\epsilon}(x)$ such that 
	Assumptions \ref{ass:optimality}, \ref{ass:optimality2} and \ref{ass:optimality3} hold, 
	and that $\nabla^2 \alpha_\epsilon(x_\ast)$ is invertible.
	Then the following assertion is satisfied by any three consecutive iterates $x^{(k-1)}, x^{(k)}, x^{(k+1)}$ 
	of Algorithm \ref{alg2}, that are sufficiently close to $x_\ast$: 
	there exists a constant $\Upsilon > 0$ independent of $k$ such that
	\[
	\frac{ \big\| x^{(k+1)} - x_\ast  \big\|_2 }{ \big\| x^{(k)} - x_\ast  \big\|_2 \max \{ \big\| x^{(k)} - x_\ast  \big\|_2 ,  
					\big\| x^{(k-1)} - x_\ast  \big\|_2 \}} 	\;	\leq	\;	 \Upsilon. 
	\]
\end{theorem}

Theorem \ref{thm:gconv1} related to global convergence
extend for Algorithm \ref{alg2}\footnote{To be precise, Theorem \ref{thm:gconv1} 
also holds when the occurrence of Algorithm \ref{alg1} is replaced by Algorithm \ref{alg2}
in the theorem statement.}, 
as we still attain the monotonicity 
$\alpha_\epsilon^{\cV_{\ell_1}}(x) \leq \alpha_\epsilon^{\cV_{\ell_2}}(x) \leq \alpha_\epsilon(x)$
for all $x \in \Omega$ for positive integers $\ell_1 , \ell_2$ such that $\ell_1 \leq \ell_2$, and Hermite 
interpolation at points $x^{(k)}$ with Algorithm \ref{alg2}.


\begin{algorithm}
\begin{algorithmic}[1]
	\REQUIRE{The matrix-valued function $A(x)$, the feasible region $\underline{\Omega} \,$, and $\epsilon > 0$.}
	\ENSURE{An estimate $\widehat{x}$ for $\arg\min_{x\in\underline{\Omega}}\alpha_\epsilon (x)$, and $\widehat{z} \in {\mathbb C}$
				that is an estimate for a globally rightmost point in $\Lambda_{\epsilon}(\widehat{x})$}	
	\vskip .6ex
	\STATE{$x^{(1)}_1, \dots x^{(1)}_\eta \gets$ initially chosen points in $\underline{\Omega}$.}\label{alg:init_points_2}
	\STATE{$z^{(1)}_j \gets \arg\max\left\{\text{Re}(z) \; | \; z \in {\mathbb C} \, \text{ s.t. } \, 
						\sigma_{\min}(A(x^{(1)}_j)-zI)\leq \epsilon \right\}$ for $j = 1, \dots , \eta$.}\label{line:abs1_2}
	\STATE{$v^{(1)}_{j} \gets$ a right singular vector corr. to $\sigma_{\min}(A(x^{(1)}_j)-z^{(1)}_jI)$ for $j = 1, \dots, \eta$.} \label{line:sv1_2}	
	\STATE{$\mathcal{V}_1 \gets$ span$\left\{v^{(1)}_{1} , \dots v^{(1)}_{\eta}\right\} \quad$ and $\quad V_1 \gets$ an orthonormal basis for ${\mathcal V}_1$.}\label{defn:V0_2} 
	\vskip .4ex
	\FOR{$k=2,3,\dots$}
	\vskip .3ex
	\STATE{$x^{(k)}\gets\arg\min_{x\in\underline{\Omega}}\alpha_\epsilon^{{\mathcal V}_{k-1}} (x)$.} \label{siter_start2}
	\vskip .3ex
	\STATE{$z^{(k)} \gets \arg\max\left\{\text{Re}(z) \; | \; z \in {\mathbb C} \, \text{ s.t. } \, 
									\sigma_{\min}(A(x^{(k)})-zI)\leq \epsilon \right\}$.}\label{line:absk2}
	\vskip .4ex
	\STATE{\textbf{Return} $\widehat{x} \gets x^{(k)}$, $\widehat{z} \gets z^{(k)}$ if convergence occurred.} \label{alg:terminate2}
	\vskip .75ex
	\STATE{$v^{(k)} \gets$ a unit right singular vector corresponding to $\sigma_{\min}(A(x^{(k)})-z^{(k)}I)$.}\label{line:svk2} 
	\vskip .6ex
	\STATE{$h^{(k)} \gets \| x^{(k)} - x^{(k-1)} \|$}
	\vskip .6ex
		\FOR{$p = 1, \dots , d$} \label{a2:in_for}
		\vskip .3ex
			\FOR{$q = p, \dots , d$}
			\vskip .3ex
\STATE{$x^{(k)}_{pq} \gets x^{(k)} + h^{(k)} e_{pq}$}
		\vskip .4ex
\STATE{$z^{(k)}_{pq} \gets \arg\max\left\{\text{Re}(z) \; | \; z \in {\mathbb C} \, \text{ s.t. } \, 
									\sigma_{\min}(A(x^{(k)}_{pq})-zI)\leq \epsilon \right\}$.}\label{line:absk3}
		\vskip .4ex
\STATE{$v^{(k)}_{pq} \gets$ a unit right singular vector corresponding 
						to $\sigma_{\min}(A(x^{(k)}_{pq})-z^{(k)}_{pq}I)$.}\label{line:svk3} 
	\vskip .4ex
			\ENDFOR
			\vskip .3ex
	\ENDFOR	\label{a2:in_for_end}
	\vskip .5ex
	\STATE
	{$V_k\gets \orth\left( [ V_{k-1} \;\; v^{(k)} \;\; v^{(k)}_{11} \;\; \dots v^{(k)}_{dd} ]\right) \quad$ {and}  
	$\quad \mathcal{V}_k \gets \Col(V_k)$.} \label{siter_end_2}
	\vskip .5ex
	\ENDFOR
\end{algorithmic}
\caption{The extended subspace framework to minimize $\alpha_{\epsilon}(x)$ over $\underline{\Omega}$}
\label{alg2}
\end{algorithm}

\section{Extensions to Real Pseudospectral Abscissa}\label{sec:r_psa}
The real $\epsilon$-pseudospectrum, and real $\epsilon$-pseudospectral
abscissa of $A \in {\mathbb C}^{n\times n}$ for a given $\epsilon > 0$ are defined by
\bigskip
\begin{center}
$\,
	\Lambda^{\mathbb R}_{\epsilon}(A)
			\,	:=	\,
	\left\{
			z	\in	\C 
				\; | \; 
		z \in \Lambda(A+\Delta) \;\; 
				\exists \Delta\in {\mathbb R}^{n \times n} \; \text{ s.t. } \norm{\Delta}_2 \leq \epsilon \right\}, \,
$
\end{center}
\bigskip
and 
$\alpha^{\mathbb R}_{\epsilon}(A)$ $:= \max\left\{\Real(z) \; | \; z\in \Lambda^{\mathbb R}_{\epsilon}(A) \right\}$.
The $\epsilon$-level-set characterizations \cite{Bernhardsson1998}
\[
	\Lambda^{\mathbb R}_{\epsilon}(A)
			\;	=	\;
		\{ z \in {\mathbb C} \: | \: \mu(z) \leq \epsilon	\},
	\quad\;
	\alpha^{\mathbb R}_{\epsilon}(A)
			\;	=	\;
		\max\{  \Real(z) \: | \: z \in {\mathbb C} \;\,  \text{s.t.}  \;\,  \mu(z) \leq \epsilon	\}	\, ,
\]
where
\begin{equation}\label{eq:red_rpsa}
\begin{split}
	&
	\mu(z)
		:=
	\max_{\gamma \in (0,1]}	\:	g(z, \gamma)	\,	, 	\;\;
		\\
	&		
	g(z, \gamma)
		:=
	\sigma_{-2}
	\left(
		G(z, \gamma)
	\right)	\,	,	\quad	
	G(z, \gamma) :=
		\left[
			\begin{array}{cc}
				A - \Real(z) I		&	-\gamma \Imag(z) I	\\
				\gamma^{-1} \Imag(z) I	&	A - \Real(z) I	
			\end{array}
		\right]	\,	,	
\end{split}
\end{equation}
and $\sigma_{-2}(\cdot)$ denotes the second smallest singular value of its matrix argument
are useful for computational purposes. Fixed-point iterations based on low-rank dynamics
are proposed in the literature for the computation of $\alpha^{\mathbb R}_{\epsilon}(A)$
\cite{Guglielmi2013b, Rostami2015}. Moreover, a criss-cross type algorithm, and a subspace
method for the computation of $\alpha^{\mathbb R}_{\epsilon}(A)$ are introduced
in \cite{Lu2017}.

For a matrix-valued function $A(x)$ defined as in (\ref{eq:affine_mvf}) and
a prescribed $\epsilon > 0$,
it may be of interest to minimize $\alpha^{\mathbb R}_{\epsilon}(A(x))$
over $x$ in a compact set $\underline{\Omega}$ when the uncertainties
are constrained to be real.
A direct extension of Algorithm \ref{alg1}
would minimize a reduced counterpart of $\alpha^{\mathbb R}_{\epsilon}(A(x))$
such as the one in \cite[Section 6]{Lu2017}, and locate the rightmost point $z_\ast$ in
$\Lambda^{\mathbb R}_{\epsilon}(A(x_\ast))$ at the minimizer $x_\ast$ of the
reduced problem. Then the subspace defining the reduced
problem can be expanded with the inclusion of directions defined in terms
of the singular vectors corresponding to the smallest two singular values of
$G(z_\ast , \gamma_\ast)$ as in (\ref{eq:red_rpsa}) with $A = A(x_\ast)$, where
$\gamma_\ast$ is the maximizer of $g(z_\ast,\gamma)$ over $\gamma$.
The monotonicity properties analogous to the ones in Lemma \ref{monotonicity}
are immediate as argued in \cite[Section 6]{Lu2017}.
However, it appears that the interpolation properties are lost with such
a subspace framework; it seems that the interpolation properties do not
have to hold even between the function $\mu$ in (\ref{eq:red_rpsa})
and its reduced counterpart, as well as between the derivative of $\mu$
and that of the reduced counterpart. This is because $\mu$ is the maximum 
of a singular value function rather than merely a singular value function, i.e.,
even if $g(z_\ast,\gamma)$ is interpolated by its reduced counterpart at $\gamma = \gamma_\ast$,
the reduced counterpart of $g(z_\ast,\gamma)$ is not necessarily maximized at $\gamma_\ast$,
implying that the reduced counterpart of $\mu(z)$ may be larger than $\mu(z)$ at $z = z_\ast$.
Extension of Algorithm \ref{alg1} for the large-scale minimization of the real
pseudospectral abscissa, or at least analyzing such an extension,
is not straightforward, and goes beyond the scope of this work.

\section{Numerical Results}\label{sec:num_results}
We have implemented the proposed subspace framework, that is Algorithm \ref{alg1}, 
to minimize $\alpha_{\epsilon}(x)$ in Matlab.
In this section, we perform numerical experiments using this implementation  in Matlab 2020b on an 
iMac with Mac OS~12.1 operating system, Intel\textsuperscript{\textregistered} Core\textsuperscript{\texttrademark} 
i5-9600K CPU and 32GB RAM. Our implementation of the framework 
seems to be converging rapidly even for multiparameter examples.
Before focusing on numerical experiments on synthetic examples
and benchmark examples taken from the \emph{COMP}$l_e ib$ collection \cite{Leibfritz2004}
in Sections \ref{sec:num_syth} and \ref{sec:num_benchmark}, we first summarize a few important
implementation details in the next subsection.

\subsection{Implementation Details}\label{sec:num_implement}

\subsubsection{Initial Interpolation Points}
We choose the initial interpolation points  $x^{(1)}_1, \dots, x^{(1)}_\eta$ in line \ref{alg:init_points} of Algorithm \ref{alg1}
as follows unless they are made available explicitly.  If there are multiple optimization parameters, then
$x^{(1)}_1, \dots, x^{(1)}_\eta$ are selected randomly in $\underline{\Omega}$. Otherwise, if there is only one optimization
parameter, we choose them as equally-spaced points in $\underline{\Omega} := [ L , U]$, i.e., 
$x^{(1)}_j  := L  +  (j-1) \frac{U - L}{\eta - 1}$ for $j = 1, \dots , \eta$. In all of the experiments below, unless otherwise 
stated, the number of initial interpolation points is $\eta = 10$.

\subsubsection{Termination Condition}\label{sec:num_terminate}
In practice we terminate when the gap between the optimal values for the reduced optimization problems 
in two consecutive iterations is less than a prescribed tolerance. Formally, given
a tolerance \texttt{tol}, we terminate at iteration $k \geq 3$ if
\begin{equation}\label{eq:terminate}
	\alpha_\epsilon^{{\mathcal V}_{k-1}} (x^{(k)})	-	\alpha_\epsilon^{{\mathcal V}_{k-2}} (x^{(k-1)})
			\;\;	<	\;\;	\texttt{tol}.
\end{equation}
In practice we use \texttt{tol} $= \: 10^{-7}$. Additionally, we terminate if the number of
subspace iterations exceeds a certain amount, which is 30 iterations in all of the  numerical
experiments below. But this second condition is never needed for the examples here
as the condition in (\ref{eq:terminate}) with \texttt{tol} $= \: 10^{-7}$ is always met in
fewer than 10 iterations.

\subsubsection{Solution of the Reduced Optimization Problems}\label{sec:num_red_prob}
The minimization of $\alpha_\epsilon^{{\mathcal V}_{k-1}} (x)$ over $x \in \underline{\Omega}$ in line \ref{siter_start} 
of Algorithm \ref{alg1} is performed using either ``eigopt'' \cite{Mengi2014} or ``GRANSO'' \cite{Curtis2017}.

If there is only one parameter, then we rely on ``eigopt'' as it converges globally provided a lower bound $\gamma$
on the second derivatives of $\alpha_\epsilon^{{\mathcal V}_{k-1}} (x)$ where it is differentiable is chosen small enough.
In all of the experiments below depending on one parameter, we set $\gamma = -400$, which seems to work well.

On the other hand, if there are multiple parameters, we depend on ``GRANSO'' for the solution of these
reduced optimization problems. The reason is ``eigopt'' is usually slow when there are multiple optimization
parameters, and the number of function evaluations needed to satisfy a certain accuracy increases 
quite rapidly with respect to the accuracy required. ``GRANSO'', on the other hand, can solve multiple-parameter
problems quite efficiently. Yet, as it is based on BFGS (i.e., quasi-Newton methods), it can converge to
a local minimizer that is not a global minimizer. For the numerical examples depending on multiple 
parameters that we report below, this local convergence issue does not cause any problem.

\subsubsection{Objective for the Reduced Optimization Problems and Rectangular Pseudospectra}\label{sec:obj_red_prob}
The objective for the reduced optimization problem at the subspace iteration with loop index $k$ is  
$\alpha_\epsilon^{{\mathcal V}_{k-1}} (x)$ meaning $\alpha_\epsilon^{{\mathcal V}_{k-1}} (x)$
needs to be computed at several $x$. This requires finding the rightmost point of the
rectangular pseudospectrum
\[
	\Lambda_{\epsilon}(A^{V_{k-1}}(x))	\;	=	\;	
	\{ z \in {\mathbb C} \; | \; \sigma_{\min} ( A^{V_{k-1}}(x)  -  z V_{k-1}) \leq \epsilon \}.
\]
In general, finding the rightmost point of a rectangular psedospectrum of the form
\[
	\{ z \in {\mathbb C} \; | \; \sigma_{\min}(E - z F) \leq \epsilon	\}
\]
for given matrices $E, F \in {\mathbb C}^{p\times q}$ with $p > q$ is a challenging
problem. For instance, it does not seem easy to determine how many connected components 
such a rectangular pseudospectrum has, and where the connected components are in
the complex plane.

To find the rightmost point of $\Lambda_{\epsilon}(A^{V_{k-1}}(x))$, 
we adopt the approach proposed in \cite[Section 5.2]{Kressner2014}, which is
an extension of the quadratically convergent criss-cross algorithm for the $\epsilon$-pseudospectral
abscissa of a square matrix \cite{Overton2003} to rectangular pencils. We remark that $A^{V_{k-1}}(x)$ and $V_{k-1}$
are of size $n\times \widetilde{k}$, where $\widetilde{k} := k-2+\eta$. Hence, at first look they are not small scale. Yet, as suggested in \cite[Section 5.2]{Kressner2014},
a reduced QR factorization
\[
	\left[
		\begin{array}{cc}
			V_{k-1}		&		A^{V_{k-1}}(x)
		\end{array}
	\right]
			\;	=	\;
		Q	
	\left[
		\begin{array}{cc}
			\widetilde{B}		&		\widetilde{A}
		\end{array}
	\right]		
\]
yields $\widetilde{A}, \widetilde{B} \in {\mathbb C}^{2\widetilde{k} \times \widetilde{k}}$ such that
\[
	\Lambda_{\epsilon}(A^{V_{k-1}}(x)) 	\: = \:		\Lambda_{\epsilon}(\widetilde{A}, \widetilde{B})	:=	
		\{ z \in {\mathbb C} \; | \; \sigma_{\min} ( \widetilde{A}  -  z \widetilde{B}) \leq \epsilon \}. 
\]
To summarize, to compute $\alpha_\epsilon^{{\mathcal V}_{k-1}} (x)$ the rightmost point of 
$\Lambda_{\epsilon}(\widetilde{A}, \widetilde{B})$ depending on small matrices is found
cheaply using a criss-cross algorithm as in \cite{Overton2003}.

One subtle issue is that it is essential for the criss-cross algorithm to locate the rightmost
point in $\Lambda_{\epsilon}(\widetilde{A}, \widetilde{B})$ to start iterating initially from 
a point in the rightmost connected component of $\Lambda_{\epsilon}(\widetilde{A}, \widetilde{B})$.
As a heuristic, it is proposed in \cite[Section 5.2]{Kressner2014} to initialize the criss-cross
algorithm with the rightmost eigenvalue $\lambda$ of the pencil $L(s) = \widetilde{A}_1 - s \widetilde{B}_1$ 
satisfying $\sigma_{\min} (\widetilde{A} - \lambda \widetilde{B}) \leq \epsilon$, where $\widetilde{A}_1, \widetilde{B}_1$
denote the upper $\widetilde{k} \times \widetilde{k}$ parts of $\widetilde{A}, \widetilde{B}$. The difficulty is
that there may not be such an eigenvalue of $L(\cdot)$ satisfying 
$\sigma_{\min} (\widetilde{A} - \lambda \widetilde{B}) \leq \epsilon$.
As a safeguard, we additionally choose equally-spaced points $y_1, \dots , y_s$ in a prescribed 
subinterval (for the examples below $[-2{\rm i} , 2 {\rm i}]$) of the imaginary axis, then find the
largest $x$ such that $\sigma_{\min}( \widetilde{A} - (x + {\rm i} y_j ) \widetilde{B}) = \epsilon\,$ for $j = 1, \dots, s$. 
Such largest $x$ corresponding to $y_j$, call it $x^{LR}_{j}$, for $j = 1, \dots, s$ is given by the largest
imaginary part attained over the purely imaginary eigenvalues of the pencil \cite[Lemma 5.1]{Kressner2014}
\[
		{\mathcal L}(s)	
			\;	=	\;
		\left[
			\begin{array}{cc}
				-y_j \widetilde{B}^\ast + {\rm i} \widetilde{A}^\ast	&	\epsilon I		\\
				-\epsilon I		&		y_j \widetilde{B} + {\rm i} \widetilde{A}
			\end{array}
		\right]	
			\:	-	\:
		s
		\left[
			\begin{array}{cc}
				\widetilde{B}^\ast	&	0	\\
				0		&		\widetilde{B}	
			\end{array}
		\right]	\,	.
\]
To conclude, we initialize the criss-cross algorithm for rectangular pencils with the 
rightmost point among the points produced by the heuristic described above from \cite[Section 5.2]{Kressner2014}
and $x^{LR}_{j} + {\rm i} y_{j}$ for $j = 1, \dots , s$.

\subsubsection{Finding the Rightmost Point of $\Lambda_{\epsilon}(x)$}
The computation of the rightmost point of 
$\Lambda_{\epsilon}(x^{(1)}_j)$ for $j = 1,2, \dots , \eta$, and 
$\Lambda_{\epsilon}(x^{(\ell)})$ for $\ell = 2, \dots , k$ is required by the subspace framework
to form the subspace ${\mathcal V}_k$. This is usually the most intensive computational task, as it involves
finding the rightmost point in the $\epsilon$-pseudospectral abscissa of the large matrix $A(x)$ at various $x$. 
For this purpose, we either employ the original criss-cross algorithm for the computation of the pseudospectral
abscissa \cite{Overton2003} if the size of $A(x)$ is less than or equal to a prescribed amount, 
or otherwise the subspace framework in \cite{Kressner2014} if the size of $A(x)$ is larger than the prescribed
amount. In all of the experiments below, this prescribed size is chosen as 1000.

\subsection{Synthetic Examples Depending on One Parameter}\label{sec:num_syth}
We first conduct numerical experiments with the synthetic examples used in \cite[Section 7]{Mengi2018},
that are publicly available\footnote{\url{http://home.ku.edu.tr/~emengi/software/max_di/Data_&_Updates.html}}.
These examples concern the distance to instability, which, for a matrix $M \in {\mathbb C}^{n\times n}$, 
is defined by
\[
	\mathcal{D}(M)		\;	:=	\;	\inf \{ \: \| \Delta \|_2 \; | \;	\Delta \in {\mathbb C}^{n\times n}	\text{ s.t. }		\Lambda(M + \Delta ) \cap \mathbb{C}^+	\neq \emptyset	 \: \},
\]
where $\mathbb{C}^+$ denotes the closed right-half of the complex plane. It follows from the definition
of $\mathcal{D}(M)$ that, for every $\epsilon > 0$, we have
\begin{equation}\label{eq:dist_psa}
	\mathcal{D}(M)	\leq	\epsilon		\quad	\Longleftrightarrow		\quad		\alpha_\epsilon(M) \geq 0.
\end{equation}
Each example in \cite[Section 7]{Mengi2018} involves a matrix-valued function of the form
$A(x) = A + x b c^T$ over $x \in {\mathbb R}$ in a prescribed interval for given $A \in {\mathbb R}^{n\times n}$,
$b, c \in {\mathbb R}^n$. Specifically, the distance to instability for each example $A(x)$ is maximized over $x$
in a prescribed interval ${\mathcal I}$, and the maximal value of the distance to instability, as well as the
maximizer are reported. Letting $x_\ast$ be the global maximizer of ${\mathcal D}(A(x))$ over $x \in {\mathcal I}$, 
it is immediate from (\ref{eq:dist_psa}) that
$\,
\min_{x \in {\mathcal I}} \, \alpha_{\epsilon}(A(x))	\,	=	\,	\alpha_\epsilon(A(x_\ast))	 \, =	\, 0 \,
$
for $\, \epsilon =  {\mathcal D}(A(x_\ast))$.

We illustrate the proposed subspace framework  in Figure \ref{fig:x_vs_psa} to minimize $\alpha_\epsilon(x) = \alpha_\epsilon(A(x))$ for the 
example $A(x)$ in \cite[Section 7]{Mengi2018} of order $n = 400$ over $x \in [-0.3,0.2]$, and with 
$\epsilon = 0.12870882$, the reported value of ${\mathcal D}(A(x_\ast))$ in \cite[Section 7]{Mengi2018}. The initial 
subspace ${\mathcal V}_1$ is chosen as the 2-dimensional 
subspace so that Hermite interpolation is attained between $\alpha_{\epsilon}(x)$ and $\alpha^{\mathcal{V}_1}_\epsilon(x)$ at $x = -0.2, 0.1$. 
The reduced function $\alpha^{\mathcal{V}_1}_\epsilon(x)$ is minimized over $x \in [-0.3,0.2]$; its minimizer turns out to be $x^{(2)} = -0.0480905$.
Then the subspace $\mathcal{V}_1$ is expanded into ${\mathcal V}_2$ so that $\alpha^{\mathcal{V}_2}_\epsilon(x)$ interpolates
$\alpha_{\epsilon}(x)$ at $x = x^{(2)}$. The global minimizer of $\alpha^{\mathcal{V}_2}_\epsilon(x)$ is already quite close to the
actual minimizer of $\alpha_\epsilon(x)$ as can be seen in Figure \ref{fig:x_vs_psa}. The subspace framework on this example
terminates after 4 subspace iterations. The iterates of the subspace framework are given in Table \ref{table:s400_conv}. It appears
from the second and third columns of this table that $x^{(k)}$ converges to the minimizer $x_\ast$ 
at a superlinear rate, consistent with the superlinear convergence assertion of Theorem \ref{thm:super_converge}.
Also, as expected, the globally smallest value of $\alpha_\epsilon(A(x))$ is about 0, and the computed global minimizer 
$x_\ast = -0.1056316$ is about the same as the global maximizer of ${\mathcal D}(A(x))$ reported in \cite[Section 7]{Mengi2018}.

Next we apply the subspace framework to minimize $\alpha_{\epsilon}(x)$ for the same example 
but for several values
of $\epsilon$ ranging from $10^{-6}$ to 1. The computed minimal values and minimizers of $\alpha_{\epsilon}(x)$
are listed with respect to $\epsilon$ in Table \ref{table:psa_vs_epsln}. The optimal values appear to be in
harmony with $\alpha_{\epsilon}(x) = 0$ for $\epsilon$ equal to the maximal distance to instability attainable,
which is about 0.12870882.
For smaller values of $\epsilon$, the optimal $\epsilon$-pseudospectral abscissa is negative, whereas,
for larger $\epsilon$ values, it is positive as expected.  We especially remark that the subspace framework
seems to work well for small values of $\epsilon$ such as $10^{-6}$.

\begin{figure}[H]
 \centering
		\hskip -3ex
			\includegraphics[width = .62\textwidth]{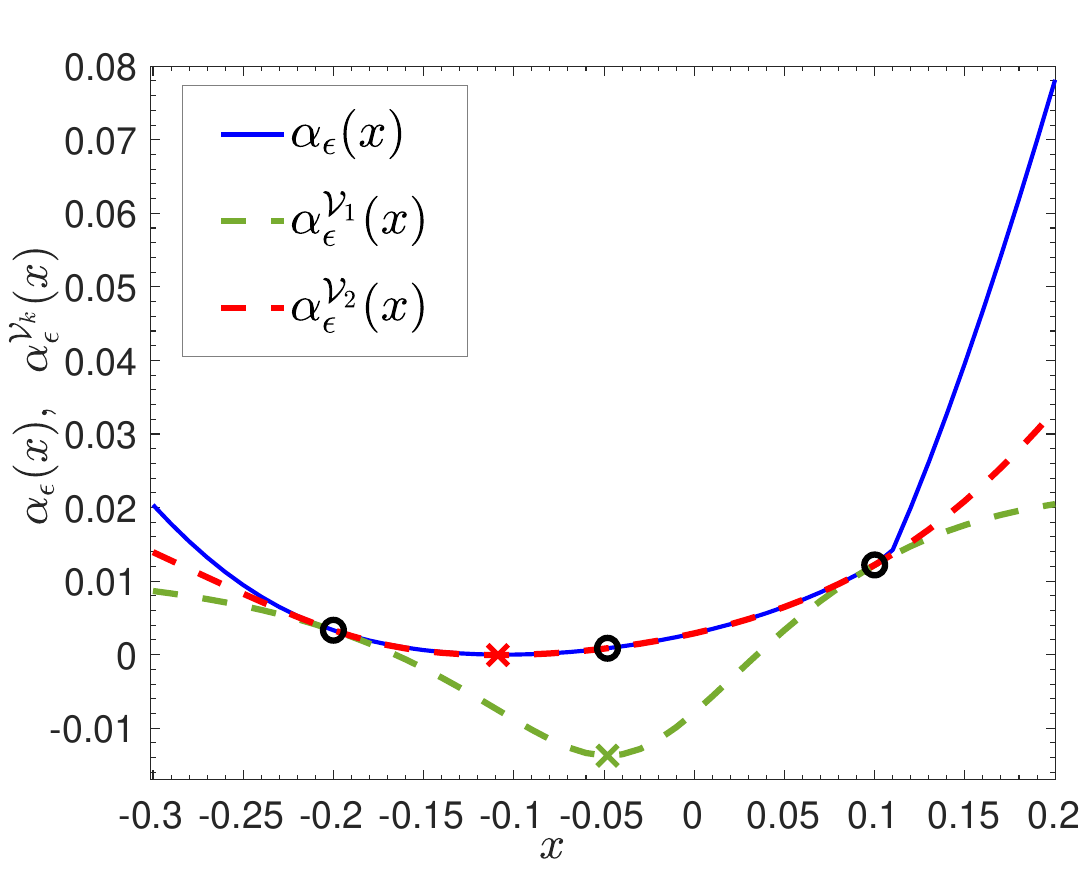} 
		\caption{  The progress of the subspace framework to minimize $\alpha_\epsilon(A(x))$
		over $x \in [-0.3, 0.2]$ for the example  $A(x)$ in \cite[Section 7]{Mengi2018} of order $n = 400$,
		and with $\epsilon$ equal to the computed value of the maximum of ${\mathcal D}(A(x))$ 
		over $x \in [-0.3,0.2]$. The black circles mark the interpolation points,
		while the crosses mark the global minimizers of $\alpha_\epsilon^{{\mathcal V}_1}(x)$ and
		$\alpha_\epsilon^{{\mathcal V}_2}(x)$.
		}
		\label{fig:x_vs_psa}
\end{figure}


\begin{table}[H]
\centering
\caption{ The iterates of the subspace framework to minimize $\alpha_\epsilon(A(x))$
for the example  $A(x)$ in \cite[Section 7]{Mengi2018} of order $n = 400$, and with $\epsilon$
equal to the computed value of the maximum of ${\mathcal D}(A(x))$ over $x \in [-0.3,0.2]$.}
\label{table:s400_conv}
\begin{tabular}{c|ccc}
	$k$  		& \phantom{aa} 	$x^{(k+1)}$ 		&  \phantom{aa}  $ |x^{(k+1)} - x_\ast |$  & 	\phantom{aaaa} $\alpha_\epsilon^{{\mathcal V}_{k}}(x^{(k+1)})$  \\    
	\hline
	1 		& \phantom{aa}  $-$\underline{0}.04808976223  	& \phantom{aa}  0.05754188697		& 	\phantom{aa} $-$\underline{0.0}1372736840 	\\
	2 		&  \phantom{aa} $-$\underline{0.10}883718893 	 &  \phantom{aa} 0.00320553973		& 	\phantom{aa} $-$\underline{0.0000}5010004	\\
	3 		& \phantom{aa}  $-$\underline{0.105631}18290	&  \phantom{aa}  0.00000004663		&	\phantom{aa} $-$\underline{0.0000000}1017	\\
	4 		& \phantom{aa} $-$\underline{0.10563164920}	&  \phantom{aa}  0.00000000000	&	\phantom{aa} $-$\underline{0.00000000753}
\end{tabular}
\end{table}




\begin{table}[H]
\centering
\caption{ The application of the subspace framework to minimize the $\epsilon$-pseudospectral
abscissa of the example $A(x)$ from \cite[Section 7]{Mengi2018} of order $n = 400$ on the interval $[-0.3,0.2]$
for various values of $\epsilon$. The global minimizer $x_\ast$ of $\alpha_{\epsilon}(x) =  \alpha_\epsilon(A(x))$ 
over $x \in [-0.3,0.2]$, 
corresponding minimal value of $\alpha_{\epsilon}(x)$, and number of subspace iterations are listed.}
\label{table:psa_vs_epsln}
\begin{tabular}{cc}
\begin{tabular}{c|ccc}
	$\epsilon$ 		&  	$x_\ast$		&		$\;\;\; \alpha_\epsilon (x_\ast)$		&	iter			\\    
	\hline
 	$10^{-6}$			&	$-$0.26107			&	$-$0.4381593	& 4 		\\			
	$10^{-5}$			&	$-$0.26114			&	$-$0.4380593	& 4 	 	\\			
	$10^{-4}$			&	$-$0.26181			&	$-$0.4370651	& 4		\\			
	$10^{-3}$			&	$-$0.26742			&	$-$0.4276521	& 4	 	\\
	$10^{-2}$			&	$-$0.27941			&	$-$0.3620685	& 3		\\
	$10^{-1}$			&	$-$0.07932			&	$-$0.0731799	& 2	
\end{tabular}
\,
		&	
\begin{tabular}{c|ccc}
	$\epsilon$ 		&  	$x_\ast$		&		$\;\;\; \alpha_\epsilon (x_\ast)$		&	iter			\\    
	\hline
 	$0.12$			&	$-$0.09836			&	$-$0.0215016			& 2 		\\			
	$0.2$			&	$-$0.15639			&	\phantom{$-$}0.1602082	& 2 	 	\\			
	$0.4$			&	$-$0.19641			&	\phantom{$-$}0.5285811	& 2		\\			
	$0.6$			&	$-$0.14273			&	\phantom{$-$}0.8384027	& 2	 	\\
	$0.8$			&	$-$0.11798			&	\phantom{$-$}1.1192159	& 2		\\
	$1$				&	$-$0.10290			&	\phantom{$-$}1.3831928	& 2	
\end{tabular}
\end{tabular}
\end{table}


We have also performed experiments with the examples from \cite[Section 7]{Mengi2018} 
of order $n = 200, 400, 800, 1200, 2000$
that concern the maximization of ${\mathcal D}(A(x))$ on the interval $[-3, 3]$, where $A(x) = A + xbc^T$ for given 
$A \in {\mathbb R}^{n\times n}, b,c \in {\mathbb R}^n$. In each case, we have minimized 
$\alpha_\epsilon (A(x))$ over $x \in [-3,3]$
using the subspace framework for $\epsilon$ equal to the reported maximal value of ${\mathcal D}(A(x))$ 
in \cite[Section 7]{Mengi2018}, which is eight decimal digit accurate.
The results are listed in Table \ref{table:syth_results}. The globally minimal value $\alpha_\epsilon(x_\ast)$ of $\alpha_\epsilon(x)$
is about 0 for each one of these examples as expected. Moreover, the computed global minimizers $x_\ast$ 
of $\alpha_{\epsilon}(x)$
listed in the table are close to those reported in \cite[Section 7]{Mengi2018}. The number of iterations until termination
for each one of the examples is 3 to 6 indicating quick convergence. As for the runtimes, according to the table, the main 
computational task that contributes to the runtime is the computation of the $\epsilon$-pseudospectral abscissa
of $A(x)$, which is required once per iteration as well as to form the initial subspaces. On the other hand, the reduced
optimization problem that involves the minimization of $\alpha_{\epsilon}^{{\mathcal V}_{k-1}}(x)$ at the subspace
iteration with loop index $k$ is relatively cheap to solve. These two computational tasks determine the total runtime. 

The runtimes for
a direct minimization (i.e., via \texttt{eigopt}) of $\alpha_{\epsilon}(x)$ without the subspace framework for minimization
are given in the last column of Table \ref{table:syth_results}. Here, $\alpha_{\epsilon}(x)$ is the objective, and needs
to be computed for several values of $x$. Following the practice used in the subspace framework, $\alpha_{\epsilon}(x)$
is computed directly using the criss-cross algorithm \cite{Burke2003} if the size of $A(x)$ is at most 1000 
(i.e., in the first three cases in the table with $n = 200, 400, 800$). On the other hand, if the size of $A(x)$
is larger than 1000, then the subspace framework in \cite{Kressner2014} is employed to compute $\alpha_{\epsilon}(x)$.
In all cases in the table, the direct minimization of $\alpha_{\epsilon}(x)$ takes considerably more time as 
compared to the minimization of $\alpha_{\epsilon}(x)$ via the proposed subspace framework.

\begin{table}
\centering
\caption{ The table concerns the application of the subspace framework to the examples from \cite[Section 7]{Mengi2018}
of order $n = 200, 400, 800, 1200, 2000$. In each case, we minimize 
$\alpha_{\epsilon}(x) = \alpha_{\epsilon}(A(x))$ over $x \in [-3,3]$,
where $\epsilon$ is the reported largest value of ${\mathcal D}(A(x))$ over $x \in [-3,3]$. 
The computed minimizers of $\alpha_{\epsilon}(x)$ are listed in the column of $x_\ast$.
The last four columns
list the total runtime (time), time for the reduced minimization problems (red), time to compute the rightmost 
point of $\Lambda_{\epsilon}(A(x))$ at various $x$ (psa) for the subspace framework, and time to solve the 
minimization problem directly without using the subspace framework (direct) in seconds.}
\label{table:syth_results}
\begin{tabular}{cc|cccccccc}
	$n$		&	$\epsilon$ 		&  $\;\;\;\; x_\ast$	&	$\alpha_\epsilon (x_\ast)$	&	iter			&  	time  	& 	 red 		&	psa		& 	direct		\\    
	\hline
	200		&	 0.01839422			&	\phantom{$-$}0.09439	&	0	& 5 	 & 	17.7		& 	7		&		10.6		&	50.6		\\			
	400 		&	0.12870882			&	$-$0.10566			&	0	& 4 	 &  	78.9	&	23.5		&		55.1		&	1547.1 	\\			
	800 		& 	0.11545563			&	$-$0.05851			&	0.0000010	& 4	&  	241.3	&	24.7		&		214.9	&	5713.7	\\	
	1200 	& 	0.07941192			&	$-$0.48625			&	0.0000003	& 3	 &     181.4	&	64.1	&		116.2	&	8732.8	\\
	2000		& 	0.08436380			&	$-$0.04766			&	0.0000005	& 6	&	737.9	&	24.5		&		708.5	&	7574.6   
\end{tabular}
\end{table}

\subsection{Benchmark Examples}\label{sec:num_benchmark}
Several benchmark examples for the stabilization by static output feedback problem are provided in the
\emph{COMP}$l_e ib$ collection \cite{Leibfritz2004}. The problems that we consider here
taken from \emph{COMP}$l_e ib$ concern finding a $K \in {\mathbb R}^{m\times p}$ such that 
$A + BKC$ has all of its eigenvalues in the open left-half of the complex plane for given
$A \in {\mathbb R}^{n\times n}$, $B \in {\mathbb R}^{n\times m}$, $C \in {\mathbb R}^{p\times n}$.
Minimizing the $\epsilon$-pseudospectral abscissa of $A + BKC$ over $K$ for a prescribed $\epsilon$ 
has also been suggested in \cite{Leibfritz2004} for the robust stabilization of the system.

We focus on three examples. The first one is about the stabilization of a single-input-single-output 
system so that $K$ is a scalar meaning there is only one parameter to be optimized, whereas
the remaining two examples depend on multiple optimization parameters. Especially, the first and third
examples involve relatively large matrices so that a direct minimization of the $\epsilon$-pseudospectral
abscissa requires considerably more time, which we report below. In all of the examples, $\epsilon$ is 
chosen as 0.2.

\medskip

\textit{\textbf{NN18} (n = 1006, m = p =1).}
This example involves the stabilization of $A + x b c^T$ over $x \in {\mathbb R}$. We minimize
$\alpha_{\epsilon}(x)$ for $\epsilon = 0.2$ over $x \in [-1,1]$ by employing the subspace framework.
The computed global minimizer is $x_\ast = -1$, and the computed globally smallest $\epsilon$-pseudospectral
abscissa is $\alpha_{\epsilon}(x_\ast) = -0.9149600$. The correctness of these computed values
can be verified by looking at Figure \ref{fig:NN18_psa}. Note that the original matrix $A$ is asymptotically
stable, yet its $\epsilon$-pseudospectral abscissa $\alpha_\epsilon(0) = -0.8$ is larger than $\alpha_{\epsilon}(x_\ast)$.
Hence, optimizing $\alpha_\epsilon(x)$ over $x$ yields a system that is more robustly stable compared
to the original system. The subspace framework terminates after 3 subspace iterations, and the iterates
generated are listed in Table \ref{table:NN18_conv}. We also report the runtimes of the subspace framework
in Table \ref{table:NN18_results}. Once again the computation of $\alpha_\epsilon(x)$ required several times
to form and expand the subspaces dominate the runtime, whereas the minimization of the reduced pseudospectral
abscissa functions is computationally much cheaper. The last column of Table \ref{table:NN18_results}
reports the total runtime in seconds for the direct minimization of $\alpha_{\epsilon}(x)$ (i.e., via \texttt{eigopt}), 
where we employ the subspace framework from \cite{Kressner2014} to compute the objective $\alpha_{\epsilon}(x)$;
the direct minimization of $\alpha_{\epsilon}(x)$ takes about 8 times
more computational time compared to that for the proposed subspace framework.

\begin{figure}
 \centering
	\begin{tabular}{ccc}
		\hskip -4.5ex
			\includegraphics[width = .5\textwidth]{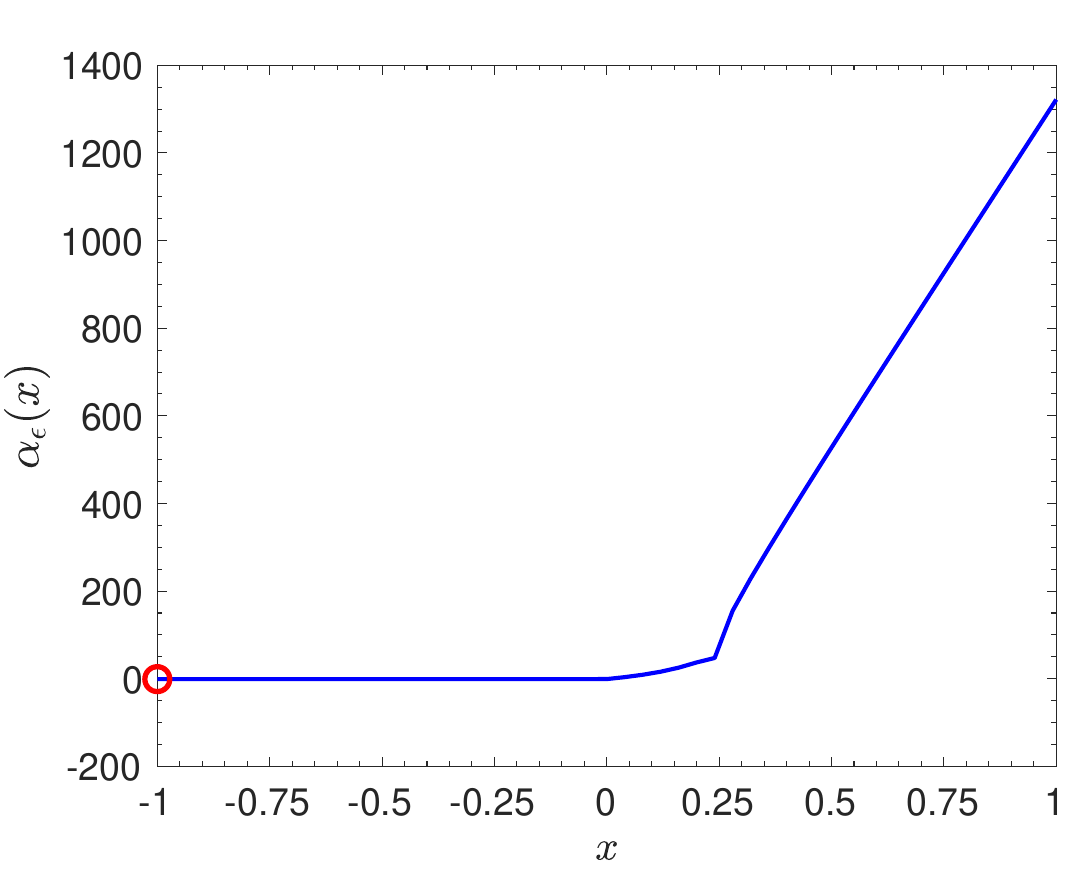} & 
			\includegraphics[width = .5\textwidth]{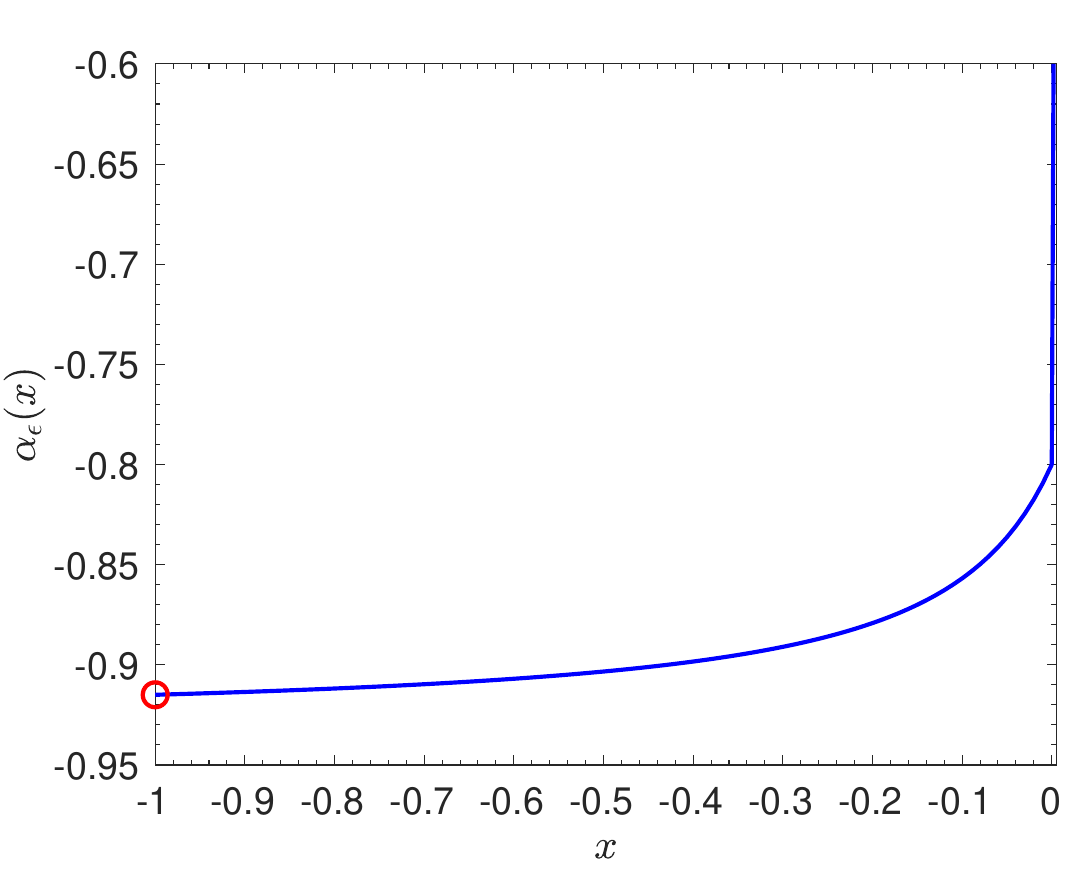} 
	\end{tabular}
		\caption{  The plots of $\alpha_{\epsilon}(x)$ as a function of $x$ for the NN18 example 
		in the \emph{COMP}$l_e ib$ collection. The right-hand plot is a zoomed version of the
		left-hand plot focusing on $x \in[-1, 0.02]$.
		}
		\label{fig:NN18_psa}
\end{figure}

\begin{table}
\centering
\caption{ The iterates of the subspace framework to minimize $\alpha_\epsilon(x)$ over $x \in [-1,1]$
for the NN18 example.}
\label{table:NN18_conv}
\begin{tabular}{c|cc}
	$k$  		& \phantom{aaaaaa} 	$x^{(k+1)}$ 			& 	\phantom{aaaa} $\alpha_\epsilon^{{\mathcal V}_{k}}(x^{(k+1)})$  \\    
	\hline
	1 		& \phantom{aa}  \phantom{$-$}0.23190  	& 	\phantom{aa} $-$1.0680310 	\\
	2 		&  \phantom{aa} $-$1.00000	 			&  	\phantom{aa} $-$0.9149600	\\
	3 		& \phantom{aa}  $-$1.00000				& 	\phantom{aa} $-$0.9149600	
\end{tabular}
\end{table}

\begin{table}
\centering
\caption{ Optimal values for $\alpha_\epsilon(x)$ over $x \in [-1,1]$ for the NN18 example obtained using the subspace framework,
as well as the runtimes for the subspace framework and direct minimization.
The last four columns list runtimes in seconds as in Table \ref{table:syth_results}.}
\label{table:NN18_results}
\begin{tabular}{ccccccc}
			 \phantom{aa$-$} $x_\ast$	&	 \phantom{aa} $\alpha_\epsilon (x_\ast)$	&	 \phantom{aa}  $\alpha_\epsilon(0)$		&  	 \phantom{aa} time  	& 	  \phantom{aa} red 		&	 \phantom{aa} psa		
			 			&	 \phantom{aa} direct		\\    
	\hline
			 \phantom{aa} $-$1	&	 \phantom{aa} $-$0.9149600	 & 	 \phantom{aa} $-$0.8		& 	 \phantom{aa} 22.9		&		 \phantom{aa} 3.9		&		 \phantom{aa} 18.2		&	\phantom{aa} 182.9
\end{tabular}
\end{table}

\textit{\textbf{HF1} (n = 130, m=1, p = 2).}
This example concerns the stabilization of $A + BKC$ with respect to $K \in {\mathbb R}^{1\times 2}$
for given $A \in {\mathbb R}^{130\times 130}, B \in {\mathbb R}^{130}, C \in {\mathbb R}^{2\times 130}$.
We minimize $\alpha_{\epsilon}(x)$ using the subspace framework for $\epsilon = 0.2$ and
$A(x) = A + x_1 B C(1,:)$ $+ x_2 \, B C(2,:)$ over $x \in [-1,1]\times[-1,1]$, where $C(j,:)$ denotes the 
$j$th row of $C$. The subspace framework terminates after two subspace iterations with the optimal value of $x$
as $x_\ast = (-0.36364, -0.26189)$, and the  corresponding minimal value of the $\epsilon$-pseudospectral 
abscissa as $\alpha_\epsilon(x_\ast) = 0.1740919$. 
The accuracy of these computed optimal values
can be verified from Figure \ref{fig:HF1_psa}. The runtimes in Table \ref{table:HF1_results}
again confirm that the total runtime is determined by the time required for the computation of $\alpha_{\epsilon}(x)$
at several $x$.

\begin{figure}
 \centering
	\begin{tabular}{ccc}
		\hskip -4.5ex
			\includegraphics[width = .5\textwidth]{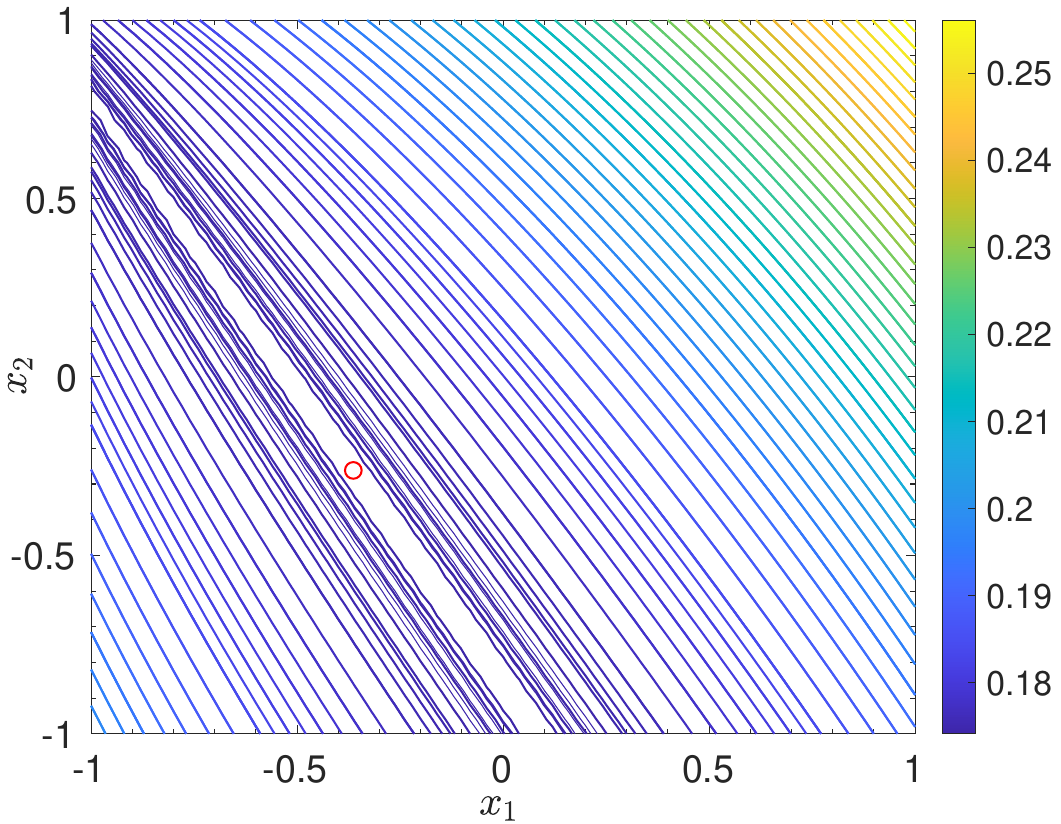} $\,$ &  
			\includegraphics[width = .53\textwidth]{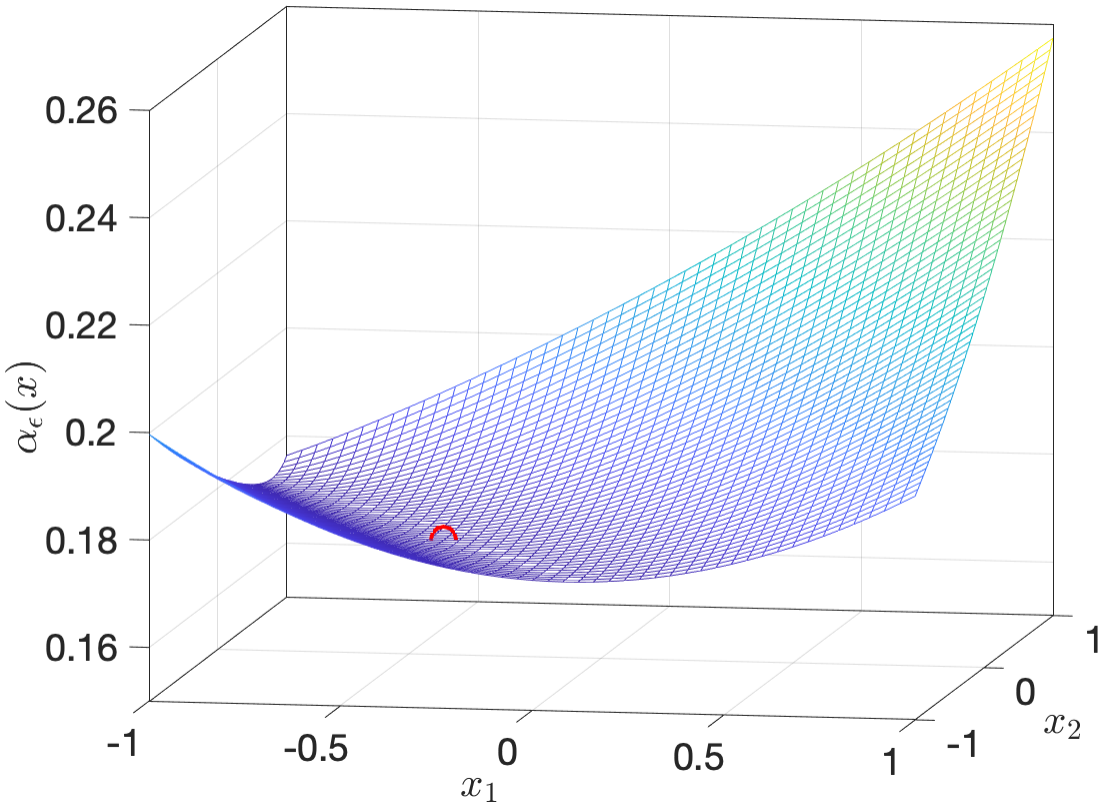} 
	\end{tabular}
		\caption{  The plots of $\alpha_{\epsilon}(x)$ for $\epsilon = 0.2$ as a function of $x$ for the HF1 example
		from the \emph{COMP}$l_e ib$ collection. In each plot, the circle marks the computed global minimizer.
		(Left) Contour diagram of $\alpha_{\epsilon}(x)$.
		(Right) 3-dimensional plot of the map $x \mapsto \alpha_{\epsilon}( x )$. 
		}
		\label{fig:HF1_psa}
\end{figure}

\begin{table}
\centering
\caption{ Optimal values for $\alpha_\epsilon(x)$ over $x \in [-1,1] \times [-1,1]$ for the HF1 example obtained using the subspace framework,
and the runtimes for the subspace framework. The last three columns list the runtimes for the subspace
framework in seconds as in Table \ref{table:syth_results}.}
\label{table:HF1_results}
\begin{tabular}{cccccc}
			 \phantom{i$-$} $x_\ast$	&	 \phantom{aa} $\alpha_\epsilon (x_\ast)$	&	 \phantom{aa}  $\alpha_\epsilon(0)$		&  	 \phantom{aa} time  	& 	  \phantom{aa} red 		&	 \phantom{aa} psa		\\    
	\hline
			 \phantom{i} ($-$0.36364, $-$0.26189)	&	 \phantom{aa} 0.1740919	 & 	 \phantom{aa} 0.1810205		& 	 \phantom{aa} 2.2		&		 \phantom{aa} 0.3		&		 \phantom{aa} 1.8		
\end{tabular}
\end{table}

\textit{\textbf{HF2D2} (n = 3796, m = 2, p = 3).}
This is a large-scale example that arises from a modeling of 2D heat flow \cite[Section 3]{Leibfritz2004}.
A stabilizer $K \in {\mathbb R}^{2\times 3}$ is sought so that $A + BKC$ is asymptotically stable for given
$A \in {\mathbb R}^{3796\times 3796} , B \in {\mathbb R}^{3796\times 2} , C \in {\mathbb R}^{3\times 3796}$.
The original matrix $A$ is unstable with spectral abscissa $0.2556862$, and $\epsilon$-pseudospectral
abscissa for $\epsilon = 0.2$ equal to $0.4625511$.

We express $A + BKC$ in the form $A(x) = A + \sum_{j=1}^6 x_j A_j$, where $x_j = k_{1j}$, $x_{3+j} = k_{2j}$,
and $A_j = B(:,1) C(j,:)$, $A_{3+j} = B(:,2) C(j,:)$ for $j = 1,2,3$, while $B(:,1)$ and $B(:,2)$ represent the 
first and second columns of $B$. We minimize $\alpha_{\epsilon}(x)$ for $\epsilon = 0.2$
with the constraints that $x_j \in [-1,1]$ for $j = 1,\dots, 6$. The subspace framework terminates after 4 subspace
iterations with optimal $x_\ast$ corresponding to the following matrix:
\[
				K_\ast
					\;	=	\;
				\left[
					\begin{array}{ccc}
						1	&	0.33494		&	1	\\
						1	&	-1			&	1	
					\end{array}
				\right].
\]
The resulting matrix $A + B K_\ast C$ is asymptotically stable;
indeed $\alpha_{\epsilon}(x_\ast) = -0.4124020$.
Figure \ref{fig:HF2D2_psa} depicts $\alpha_{\epsilon}(\widetilde{x}^{(j)})$ for five thousand randomly
chosen $\widetilde{x}^{(j)}, \, j = 1, \dots , 5000$. Every random $\widetilde{x}^{(j)}$ here is 
formed so that each one of its six components comes from a uniform distribution
over $[-1,1]$. Computing all these $\epsilon$-pseudospectral abscissa takes
about 41 minutes, yet the smallest value of the $\epsilon$-pseudospectral
abscissa over all these randomly chosen points is $-0.3092196$, significantly
larger than $\alpha_{\epsilon}(x_\ast) = -0.4124020$. As depicted in Figure \ref{fig:HF2DF_plot_psa}, whereas one of the 
components of $\Lambda_{\epsilon}(0)$ is fully on the right-hand side of the complex plane, the rightmost component 
of $\Lambda_{\epsilon}(x_\ast)$ is fully contained in the open left-half of the complex plane.
In terms of the runtime, now the reduced optimization problems take more time than the time to compute 
$\alpha_{\epsilon}(x)$ as reported in Table \ref{table:HF2DF_results}. Yet, the overall runtime for the subspace framework
is quite reasonable considering the system at hand is
of order 3796, and there are several optimization parameters. Once again the time required for the direct minimization
of $\alpha_{\epsilon}(x)$ (where again the objective $\alpha_{\epsilon}(x)$ is computed via the subspace framework
in \cite{Kressner2014}) is significantly more than the overall runtime for the subspace framework.

\begin{figure} 
 \centering
    \includegraphics[width = 1.01\textwidth]{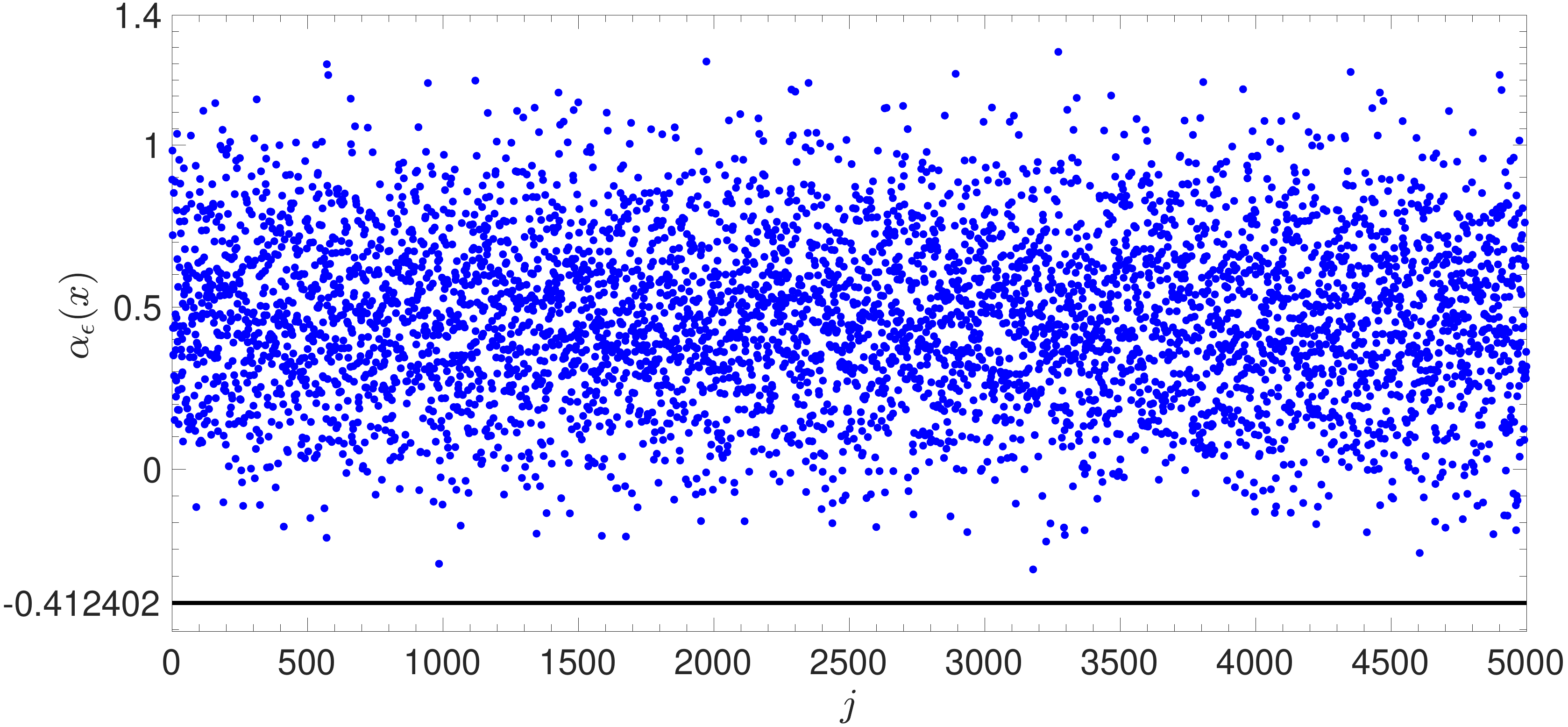} 
    \vskip -2.5ex
		\caption{  
		The plots of $\alpha_{\epsilon}(\widetilde{x}^{(j)})$ for $\epsilon = 0.2$ for the HF2D2 example 
		for randomly chosen five thousand points $\widetilde{x}^{(j)} \in [-1,1]^6$
		for $j = 1, \dots , 5000$. Every dot in the plot corresponds to $(j, \alpha_{\epsilon}(\widetilde{x}^{(j)}))$ for some $j$.   
		The horizontal line at the bottom is $y = \alpha_{\epsilon}(x_\ast) = -0.4124020$,
		where $x_\ast$ is the computed global minimizer. 
		}
		\label{fig:HF2D2_psa}
\end{figure}

\begin{figure}
 \centering
		\hskip 4ex
			\includegraphics[width = .8\textwidth]{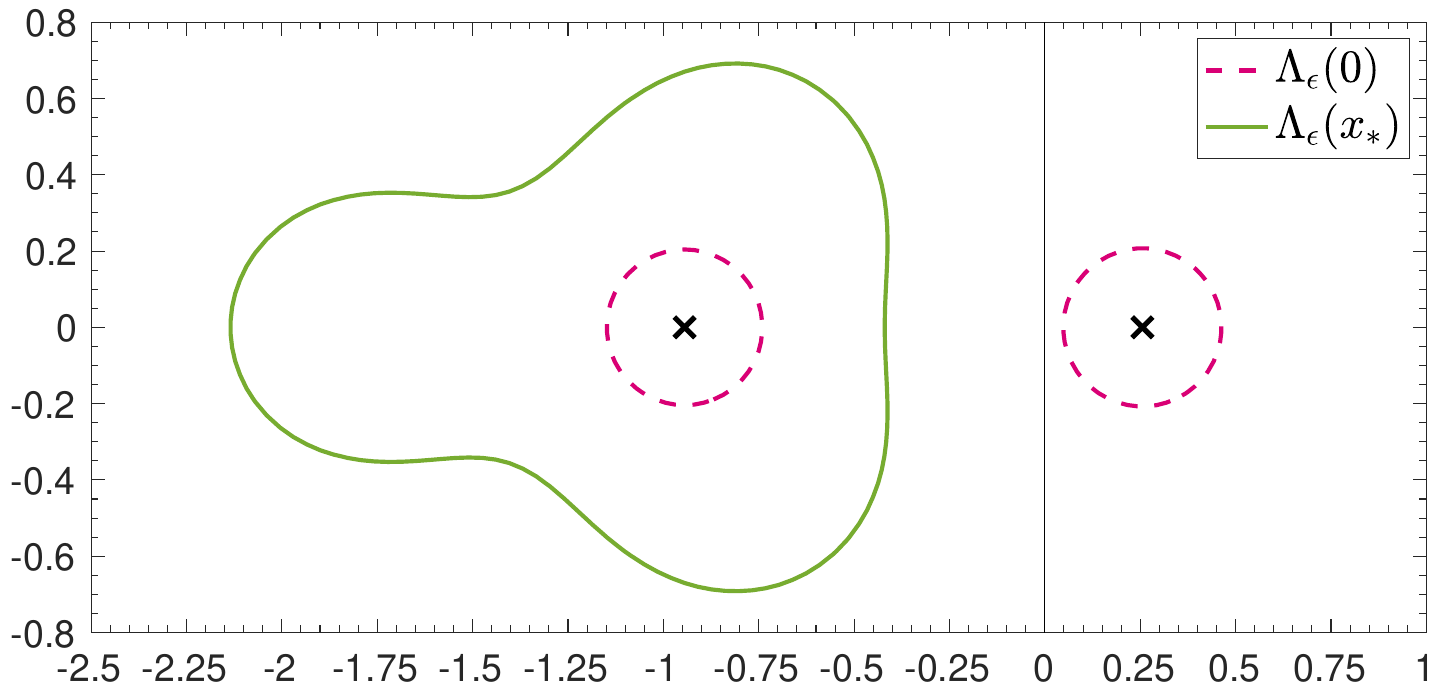} 
		\caption{  The plots of the boundaries of the rightmost components of $\Lambda_{\epsilon}(0)$ and $\Lambda_{\epsilon}(x_\ast)$
		for the HF2D2 example. The crosses mark the rightmost two eigenvalues of $A$, while the vertical
		line represents the imaginary axis. 
		}
		\label{fig:HF2DF_plot_psa}
\end{figure}

\begin{table}
\centering
\caption{ Optimal values of $\alpha_\epsilon(x)$ over $x \in [-1,1]^6$ for the HF2D2 example by the subspace framework,
and the runtimes for the subspace framework, as well as direct minimization. The last four columns are runtimes in seconds 
as in Table \ref{table:syth_results}.}
\label{table:HF2DF_results}
\begin{tabular}{cccccc}
\phantom{aa$-$} $\alpha_\epsilon (x_\ast)$	&	 \phantom{aa}  $\alpha_\epsilon(0)$		&  	 \phantom{aa} time  	& 	  \phantom{aa} red 		&	 \phantom{aa} psa	&	 \phantom{aa} direct	\\    
	\hline
		 \phantom{aa} $-$0.4124020	 & 	 \phantom{aa} 0.4625511		& 	 \phantom{aa} 36.2		&		 \phantom{aa} 25.5		&		 \phantom{aa} 9.6		&		 \phantom{aa}	211.5
\end{tabular}
\end{table}

\section{Software}
A Matlab implementation of the proposed subspace framework, that is Algorithm \ref{alg1},
is publicly available at \url{https://zenodo.org/record/6992092}. This implementation
makes use of all of the implementation details described in Section \ref{sec:num_implement}.
The numerical results on the benchmark examples in Section \ref{sec:num_benchmark}
can be reproduced by running the script \texttt{demo\_on\_benchmarks}. 


\section{Concluding Remarks}\label{sec:conclusion}
Minimization of the $\epsilon$-pseudospectral abscissa of a matrix dependent on
parameters for a prescribed $\epsilon > 0$ is motivated by robust stability and transient
behavior considerations for the associated linear control system, as well as stabilization problems 
such as the stabilization by static output feedback. Here, we have proposed a subspace framework to minimize the
$\epsilon$-pseudospectral abscissa of a large matrix-valued function dependent on parameters
analytically aiming at large-scale nature of the matrix-valued function. The large-scale
matrix-valued functions are restricted to small subspaces, and the $\epsilon$-pseudospectral
abscissa of the resulting reduced small-scale matrix-valued functions is minimized. The subspaces
are gradually expanded so as to attain Hermite interpolation properties between the
$\epsilon$-pseudospectral abscissa of the original and reduced matrix-valued functions
at the minimizers of the reduced problems. If the global minimizers of the reduced problems
stagnate, they stagnate at the global minimizer of the $\epsilon$-pseudospectral abscissa 
of the original matrix-valued function. Additionally, we have shown that the rate of convergence 
of the minimizers of the reduced problems to the global minimizer of the original problem is 
superlinear, under mild assumptions. The validity of these theoretical findings in practice 
is confirmed on synthetic and benchmark examples. The proposed framework makes it feasible 
to minimize the pseudospectral abscissa of matrix-valued functions of size on the
order of thousands, such as the NN18 and HF2D2 examples from the 
\emph{COMP}$l_e ib$ collection, in a short time.

Some of the ingredients in an actual implementation of the overall subspace framework
are locally convergent. Specifically, the subspace framework requires the rightmost point 
of the $\epsilon$-pseudospectrum of the original matrix-valued function at several
parameter values, and if the
matrix-valued function is really large, then it appears for this purpose one has to rely on a 
subspace framework such as the one proposed in \cite{Kressner2014} that converges locally.
Additionally, if the minimization is over several parameters, it appears that
the reduced minimization problems must be solved by employing a locally convergent
optimization algorithm such as those based on a quasi-Newton method, e.g., ``GRANSO'' \cite{Curtis2017}.
There is no quick remedy for these local convergence issues. However, 
on several benchmark examples considered here, the subspace framework indeed converges globally.

The reduced minimization problems require computing the $\epsilon$-pseudospectral
abscissa of rectangular pencils as in (\ref{eq:rec_psa}). We solve at the moment such 
pseudospectral abscissa problems using an extension of the criss-cross algorithm \cite{Overton2003} 
to rectangular pencils \cite[Section 5.2]{Kressner2014}. One major challenge in the
rectangular setting is that one has to start from a point in the rightmost connected component of the
associated rectangular $\epsilon$-pseudospectrum. We have outlined some ideas to overcome
this challenge in Section \ref{sec:obj_red_prob}. Yet, locating the rightmost point
in the $\epsilon$-pseudospectrum, and computing the $\epsilon$-pseudospectral abscissa 
for a rectangular pencil appear to be far from settled. This is a problem worth 
investigating in detail.


\smallskip

\textbf{Acknowledgements.} The authors are grateful to two anonymous referees,
and Mark Embree for their invaluable feedback.

\appendix

\section{Proof of Theorem \ref{eq:psa_derivs}}\label{app:psa_der_proof}
As $\widetilde{x}$ is nondegenerate, there is a unique optimal $z$ for the maximization 
problem in (\ref{eq:opt1}) with $x = \widetilde{x}$, which we denote by $z(\widetilde{x})$. Moreover,
the singular value $\sigma(\widetilde{x}, z(\widetilde{x}))$ is simple, so $\sigma(x, z)$ is
real-analytic at $(x, z) = (\widetilde{x}, z(\widetilde{x}))$.
By the first order optimality conditions applied to the constrained optimization problem 
in (\ref{eq:opt1}) with $x = \widetilde{x}$, there is $\mu(\widetilde{x})$ such that
\[
	\mathcal{L}_{z_1} (\widetilde{x},z(\widetilde{x}),\mu(\widetilde{x})) =  1 - \mu(\widetilde{x}) \sigma_{z_1}(\widetilde{x},z(\widetilde{x})) = 0.
\]
Clearly, $\mu(\widetilde{x})$ is uniquely defined by $\mu(\widetilde{x}) = 1 / \sigma_{z_1}(\widetilde{x},z(\widetilde{x}))$,
and $\mu(\widetilde{x}) \neq 0$, which in turn leads to 
	$\sigma(\widetilde{x}, z(\widetilde{x})) 	=	\epsilon$
by the complementarity conditions.

By continuity, the singular value function $\sigma(x, z)$
remains simple and real-analytic in a neighborhood of $(\widetilde{x}, z(\widetilde{x}))$.
Hence, from a standard sensitivity analysis for optimization problems, there is an open neighborhood $X$ 
of $\widetilde{x}$ such that for every $x \in X$ the optimal $z$ as well as corresponding Lagrange multiplier $\mu$
for the maximization problem in (\ref{eq:opt1}) are unique, and the unique optimizer
$y(x) = (z(x),\mu(x))$ is differentiable at every $x \in X$. From the first order optimality conditions, 
we have
\begin{equation}\label{eq:KKT}
	\nabla_y\mathcal L ( x, y(x)) = 0	\quad	\forall x \in X	.
\end{equation}

Moreover, since $\widetilde{x}$ is nondegenerate, the matrix $\nabla^2_{yy}\mathcal L ( \widetilde{x}, y(\widetilde{x}))$ is nonsingular.
Hence, the analytic implicit function theorem guarantees the existence of a unique real-analytic function
$w : U \rightarrow {\mathbb R}$ satisfying $w(\widetilde{x}) = y(\widetilde{x})$ and
$
	\nabla_y\mathcal L ( x, w(x) ) = 0	\;	\forall 	x \in U
$
in an open neighborhood $U \subseteq X \subseteq {\mathbb R}^d$ of $\widetilde{x}$. The function $w(x)$
is also the unique differentiable function satisfying $
	\nabla_y\mathcal L ( x, w(x) ) = 0	\;	\forall 	x \in U
$ and $w(\widetilde{x}) = y(\widetilde{x})$.
As a result, we have $w(x) = y(x)$ for all $x \in U$. 
Denoting the first component of $z(x) \in {\mathbb R}^2$ with $z_1(x)$,
this shows that $x \in U \mapsto \alpha_{\epsilon}(x) = z_1(x)$ is real-analytic at $x = \widetilde{x}$.

Using (\ref{eq:KKT}) as well as the analytical formulas for the derivatives of singular value functions 
(see, e.g., \cite[Section 3.3]{Mengi2014}), we deduce
\[
	0={\mathcal L}_{z_1}(\widetilde{x},z(\widetilde{x}),\mu(\widetilde{x})) 
	= 1-\mu(\widetilde{x}) \Real\left(u^*\frac{\partial}{\partial z_1} [ A(x,z) ] \bigg\vert_{\widetilde{x},z(\widetilde{x})} v\right) = 1+\mu(\widetilde{x}) \Real(u^*v)
\] 
and 
\[
	0={\mathcal L}_{z_2}(\widetilde{x},z(\widetilde{x}),\mu(\widetilde{x}))
			= 
	-\mu(\widetilde{x}) \Real\left(u^*\frac{\partial}{\partial z_2} [ A(x,z) ] \bigg\vert_{\widetilde{x},z(\widetilde{x})}v\right) = -\mu(\widetilde{x}) \Imag(u^*v),
\]
where $u,v$ consist of a pair of consistent unit left and right singular vectors corresponding to $\sigma(\widetilde{x},z(\widetilde{x})).$
The first equation implies $\Real(u^* v) \neq 0$. Furthermore, as $\mu(\widetilde{x}) \neq 0$, the second equation gives rise to $\Imag(u^*v)=0$,
that is $u^*v$ is real and so $u^* v = \Real(u^* v)$. It follows from the first equation that
\[
  		\mu(\widetilde{x})   \;  =   \;   -\frac{1}{u^*v}.
\]

From a classical sensitivity analysis result, in particular the fact that the gradient with respect
to the parameters of the
optimal value of an optimization problem dependent on parameters is equal to the gradient of the 
Lagrangian with respect to the parameters, we deduce 
\begin{equation}\label{eq:firstder2}
\nabla\alpha_\epsilon(x) = \nabla_x\cL(x,y(x))
\end{equation}
for every $x \in U$.
More specifically, denoting the $j$th component of $x$ with $x_j,$ we have
\begin{equation*}
\frac{\partial}{\partial x_j} [\alpha_\epsilon (x)] \bigg|_{\widetilde{x}}=
-\mu(x) \frac{\partial}{\partial x_j} [\sigma(x,z)] \bigg|_{\widetilde{x}, z(\widetilde{x})}.
\end{equation*}
Again, using the analytical formula for the singular value function, 
and $\mu(\widetilde{x}) = -1/ (u^* v)$, which is real and nonzero, we obtain (\ref{eq:formula_1der}).

For the second derivatives, we differentiate (\ref{eq:firstder2}) by applying the chain rule
to obtain
\begin{equation}\label{eq:secondder}
\nabla^2\alpha_\epsilon(x)   \;  =   \; 
\nabla^2_{xx}\cL(x,y(x)) + \nabla^2_{xy}\cL(x,y(x))\cdot y'(x)  \,\,   , 
\end{equation}
where $y'(\cdot)$ denotes the $3\times d \,$ Jacobian matrix of $y$ with respect to $x$.
Differentiating \eqref{eq:KKT} with respect to $x$ yields
\begin{equation*}
\nabla_{yx}^2\cL(x,y(x)) + \nabla^2_{yy}\cL(x,y(x))\cdot y'(x)
			\;\;	=	\;\;	0	\;	.
\end{equation*}
In particular, as $\widetilde{x}$ is nondegenerate, 
$\nabla^2_{yy}\cL(\widetilde{x},y(\widetilde{x}))$ is invertible, so
\begin{equation}\label{eq:jacob}
\begin{split}
y'(\widetilde{x}) 		& \; = \; 	-[\nabla^2_{yy}\cL(\widetilde{x},y(\widetilde{x}))]^{-1}\nabla^2_{yx}\cL(\widetilde{x},y(\widetilde{x})) 	 \\ 
		& \; = \; 	-[\nabla^2_{yy}\cL(\widetilde{x},y(\widetilde{x}))]^{-1}[\nabla^2_{xy}\cL(\widetilde{x},y(\widetilde{x}))]^T.
\end{split}
\end{equation}
Substituting the right-hand side of \eqref{eq:jacob} for $y'(\widetilde{x})$ in \eqref{eq:secondder} with $x = \widetilde{x}$, 
we obtain (\ref{eq:psa_2der}).
\hfill $\square$

\section{Proof of Theorem \ref{thm:notempty_psa}}\label{app:nonempty_psa_proof}
Without loss of generality, throughout this proof we assume $\| v \| = 1$.
Let $V$ be a matrix whose columns form an orthonormal basis for ${\mathcal V}$.
As $v \in {\mathcal V}$, there exists a unit vector $a$ such that $v = V a$, and so
\begin{equation*}
	\begin{split}
	\epsilon 	\,	\geq	\,
	\sigma_{\min}(A(\widetilde{x}) -  \widetilde{z} I)		\,	& =	\,	\| (A(\widetilde{x}) - \widetilde{z} I)	V a \|_2  \\
				\,	& \geq		\,	\sigma_{\min}(A^V (\widetilde{x}) - \widetilde{z} \, V).
	\end{split}	
\end{equation*}
This shows that $\widetilde{z} \in \Lambda^{\mathcal V}_{\epsilon}(\widetilde{x})$.

Indeed, from part \textbf{(i)} of Lemma \ref{monotonicity}, we also have 
$\sigma_{\min}(A(\widetilde{x}) -  \widetilde{z} I)	 \leq \sigma_{\min}(A^V (\widetilde{x}) - \widetilde{z} \, V)$,
as well as
$\sigma_{\min}(A(\widetilde{x}) -  \widetilde{z} I)	 = \epsilon$, as argued in the first paragraph of the 
proof of Theorem \ref{eq:psa_derivs} in Appendix \ref{app:psa_der_proof}. Hence, we deduce
\[
	\sigma_{\min}(A(\widetilde{x}) -  \widetilde{z} I)	= 
	\sigma_{\min}(A^V (\widetilde{x}) - \widetilde{z} \, V)  =  \epsilon .
\]

Now set $\, u := (A(\widetilde{x}) -  \widetilde{z} I) v / \epsilon \,$ so that $u$, $v$ form
a consistent pair of unit left, right singular vectors corresponding to $\sigma_{\min}(A(\widetilde{x}) -  \widetilde{z} I)$.
As $\sigma_{\min}(A(\widetilde{x}) -  \widetilde{z} I)$ is simple and nonzero, $\sigma_{\min}(A(x) -  z I)$
is differentiable at $(x, z) = (\widetilde{x}, \widetilde{z})$.
From the arguments in the proof of Theorem \ref{eq:psa_derivs}, 
the inner product $u^\ast v$ is real and nonzero.
Indeed, denoting the real part of the complex variable $z$ with $z_1$, as $\widetilde{z}$ is a rightmost point 
in $\Lambda_{\epsilon}(\widetilde{x})$, and from the formulas for the derivative of a singular value
of a matrix dependent on parameters, we must have
\[
	0	\; 	\leq	\;
	\frac{\partial}{\partial z_1}  [ \sigma_{\min}(A(x) -  z I) ] \bigg|_{\widetilde{x}, \widetilde{z}}
		\;	=	\;
	\Real
	\left(
		u^\ast
		\frac{\partial}{\partial z_1} [ A(x) -  z I ] \bigg|_{\widetilde{x}, \widetilde{z}}
		v
	\right)	
		\;	=	\;
		-u^\ast v	 .
\]
As $u^\ast v \neq 0$, we deduce $-u^\ast v > 0$.

We aim to show $\sigma_{\min}(A^V(\widetilde{x})- \widetilde{z} \,V)$ is also simple.
For the sake of contradiction, suppose otherwise.
Then there exist two unit orthonormal vectors $a,\widetilde{a}$ that satisfy
\begin{eqnarray*}
	\sigma_{\min}(A(\widetilde{x})-\widetilde{z}I)	&	=	&	
		\sigma_{\min}(A^{V}(\widetilde{x})-\widetilde{z} \,V)		\\
								&	=	&	
			\| (A^{V}(\widetilde{x})-\widetilde{z} \,V)a \|_2 = \,  
			\| (A^{V}(\widetilde{x})- \widetilde{z} \, V)\widetilde{a} \|_2.
\end{eqnarray*}
But then $\omega=V a$ and $\widetilde{\omega}=V \widetilde{a} \,$ are orthonormal vectors satisfying
\[
	\sigma_{\min}(A(\widetilde{x})- \widetilde{z} I)
		\,	=	\,
	\| (A(\widetilde{x})-\widetilde{z} I)\omega \|_2
		\,	=	\,
	\| (A(\widetilde{x})-\widetilde{z} I)\widetilde{\omega} \|_2,
\]
which contradicts with the simplicity of  $\sigma_{\min}(A(\widetilde{x})-\widetilde{z}I)$. 
Hence, $\sigma_{\min}(A^{V}(\widetilde{x})-\widetilde{z} \, V)$ must be simple.
Consequently, $\sigma_{\min}(A^{V}(x) - z V)$ is differentiable at 
$(x,z) = (\widetilde{x}, \widetilde{z})$.
	
\smallskip

Next we claim $u,a$ are a pair of consistent unit left and unit right singular vectors corresponding to 
$\sigma_{\min}(A^{V}(\widetilde{x}) - \widetilde{z} \, V) = \epsilon$. The equations 
$(A(\widetilde{x}) - \widetilde{z}I)v=\epsilon u$ and $u^\ast(A(\widetilde{x}) - \widetilde{z}I)=\epsilon v^\ast$
can be rewritten in terms of $a$ as
\[
	(A^{V}(\widetilde{x}) - \widetilde{z} \,V)a
		=
	\epsilon u \qquad \text{and} \qquad u^\ast (A^{V}(\widetilde{x}) - \widetilde{z} \,V) = \epsilon a^\ast, 
\]
confirming the claim. It follows that
\begin{equation*}
	\begin{split}
		\frac{\partial}{\partial z_1}  [ \sigma_{\min}(A^V(x) -  z V) ] \bigg|_{\widetilde{x}, \widetilde{z}}
		\;	& =	\;
	\Real
	\left(
		u^\ast
		\frac{\partial}{\partial z_1} [ A^V(x) -  z V ] \bigg|_{\widetilde{x}, \widetilde{z}}
		a
	\right)	\\
		& =	\;
	-\Real
	\left(
		u^\ast
		Va
	\right)
		\; =	\;
	- u^\ast v > 0	\:	.
	\end{split}
\end{equation*}

Letting $\sigma^{\mathcal V}(x,z) := \sigma_{\min}(A^{V}(x) - z \,V)$,
we now would like to show that $\sigma^{\mathcal V}(\widetilde{x},  \widehat{z})$ is strictly less than $\epsilon$
for some $\widehat{z} \in {\mathbb C}$ located on the left-hand side of $\widetilde{z} \in {\mathbb C}$,
by exploiting  $\sigma^{\mathcal V}(\widetilde{x},  \widetilde{z}) = \epsilon$ and 
$\frac{\partial}{\partial \, z_1} \sigma^{\mathcal V}(x,z) \, \big|_{\widetilde{x}, \widetilde{z}} > 0$
shown in the previous paragraph.
To this end, by an argument analogous to \cite[Proposition 2.9]{Kangal2018},
the second derivatives of $\sigma^{\mathcal V}(x,z)$ in a neighborhood of $(\widetilde{x}, \widetilde{z})$
are bounded from above by a constant independent of the subspace ${\mathcal V}$. Hence, 
letting $\xi := \frac{\partial}{\partial \, z_1} \sigma^{\mathcal V}(x,z) \, \big|_{\widetilde{x}, \widetilde{z}} = -u^\ast v > 0$ 
and $\widetilde{z}_1$, $\widetilde{z}_2 \,$ denote the real, imaginary parts of $\widetilde{z}$,
there is $\nu_1 \in {\mathbb R}, \nu_1 > 0$
independent of ${\mathcal V}$ such that 
\[
	\frac{\partial}{\partial z_1} [\sigma^{\mathcal V}(\widetilde{x},z)]	\geq \xi/2
	\quad \forall z =  z_1 + {\rm i} \widetilde{z}_2 \;\: \text{ s.t. } z_1 \in {\mathcal B}(\widetilde{z}_1, \nu_1) .
\] 
This in turn implies
\begin{equation}\label{eq:rsval_ub}
\sigma^{\mathcal V}(\widetilde{x},  \widehat{z})	\leq	\epsilon - \frac{\xi \nu_1}{4},
\end{equation}
where $\widehat{z} := (\widetilde{z}_1 - \nu_1/2) + {\rm i} \widetilde{z}_2$ as desired.

Moreover, $\sigma^{\mathcal V}(x,\widehat{z})$ is uniformly Lipschitz continuous with respect to $x$
\cite[Lemma 2.1]{Kangal2018}, that is there exists $\gamma \in {\mathbb R}, \gamma > 0$
independent of ${\mathcal V}$ such that
\begin{equation}\label{eq:rsval_uniLip}
	\sigma^{\mathcal V}(x,\widehat{z}) - \sigma^{\mathcal V}(\widetilde{x},\widehat{z})
		\;	\leq	\;	\gamma \cdot \| x - \widetilde{x} \|_2
		\quad	\forall x  \in \Omega	.
\end{equation}
Combining (\ref{eq:rsval_ub}) and (\ref{eq:rsval_uniLip}), we deduce
\[
	\sigma^{\mathcal V}(x,  \widehat{z})	\leq	\epsilon
	\quad	\forall x \in {\mathcal B}\big( \widetilde{x},  \frac{\xi \nu_1}{4 \gamma} \big) \cap \Omega	.
\]
Hence, $\widehat{z} \in \Lambda^{\mathcal V}_{\epsilon}(x)$
for all $x \in {\mathcal B}(\widetilde{x},  \frac{\xi \nu_1}{4\gamma}) \cap \Omega$.

Finally, suppose for every $x \in \underline{\Omega}$ there is a rightmost point 
$\zeta(x) \in \Lambda_{\epsilon}(x)$ such that 
$\sigma_{\min}(A(x) - \zeta(x) I)$
is simple. In the arguments of the previous paragraph, we now take $\xi$ as the minimum
of $-u^\ast v$ for $u, v$ that are consistent unit left, right singular vectors corresponding
to $\sigma_{\min}(A(\widetilde{x}) - \zeta(\widetilde{x}) I)$ over all $\widetilde{x} \in \underline{\Omega}$.
The quantity $\nu_1$ in (\ref{eq:rsval_ub}) and 
$\gamma$ in (\ref{eq:rsval_uniLip}) are also uniform over all $\widetilde{x} \in \underline{\Omega}$.
Hence, the arguments of the previous paragraph show that, for any $\widetilde{x} \in \underline{\Omega}$
such that a right singular vector corresponding to $\sigma_{\min}(A(\widetilde{x}) - \zeta(\widetilde{x}) I)$
is in ${\mathcal V}$, we have 
$\widehat{z}(\widetilde{x}) := (\zeta_1(\widetilde{x}) - \nu_1/2) 
		+ {\rm i} \zeta_2(\widetilde{x})  \in \Lambda^{\mathcal V}_{\epsilon}(x)$
$\: \forall x \in {\mathcal B}(\widetilde{x},  \frac{\xi \nu_1}{4\gamma}) \cap \Omega$
with $\zeta_1(\widetilde{x}), \zeta_2(\widetilde{x})$ denoting
the real, imaginary parts of $\zeta(\widetilde{x})$.




\hfill $\square$

\bibliography{na_references}

\begin{thebibliography}{10}

\bibitem{Aliyev2017}
{\sc N.~Aliyev, P.~Benner, E.~Mengi, P.~Schwerdtner, and M.~Voigt}, {\em
  Large-scale computation of $\mathcal{L}_\infty$-norms by a greedy subspace
  method}, SIAM J. Matrix Anal. Appl., 38 (2017), pp.~1496--1516.

\bibitem{Aliyev2020}
{\sc N.~Aliyev, P.~Benner, E.~Mengi, and M.~Voigt}, {\em A subspace framework
  for $\mathcal{H}_\infty$-norm minimization}, SIAM J. Matrix Anal. Appl., 41
  (2020), pp.~928--956.

\bibitem{Apkarian2006b}
{\sc P.~Apkarian and D.~Noll}, {\em Controller design via nonsmooth
  multidirectional search}, SIAM J. Control Optim., 44 (2006), pp.~1923--1949.

\bibitem{Apkarian2006}
{\sc P.~Apkarian and D.~Noll}, {\em Nonsmooth ${H}_\infty$ synthesis}, IEEE
  Trans. Autom. Control, 51 (2006), pp.~71--86.

\bibitem{Apkarian2008}
{\sc P.~Apkarian, D.~Noll, and O.~Prot}, {\em A trust region spectral bundle
  method for nonconvex eigenvalue optimization}, SIAM J. Optim., 19 (2008),
  pp.~281--306.

\bibitem{Arzelier2011}
{\sc D.~Arzelier, G.~Deaconu, S.~Gumussoy, and D.~Henrion}, {\em ${H}_2$ for
  {HIFOO}}, in Int. Conference on Control and Optimization with Industrial
  Applications, Ankara, Turkey, August 2011.

\bibitem{Benner2019}
{\sc P.~Benner and T.~Mitchell}, {\em Extended and improved criss-cross
  algorithms for computing the spectral value set abscissa and radius}, SIAM J.
  Matrix Anal. Appl., 40 (2019), pp.~1325--1352.

\bibitem{Bernhardsson1998}
{\sc B.~Bernhardsson, A.~Rantzer, and L.~Qiu}, {\em Real perturbation values
  and real quadratic forms in a complex vector space}, Linear Algebra Appl.,
  270 (1998), pp.~131--154.

\bibitem{Blondel1995}
{\sc V.~Blondel, M.~Gevares, and A.~Lindquist}, {\em Survey on the state of
  systems and control}, European J. Control, 1 (1995), pp.~5--23.

\bibitem{Blondel1997}
{\sc V.~Blondel and J.~N. Tsitsiklis}, {\em {N}{P}-hardness of some linear
  control design problems}, SIAM J. Control Optim., 35 (1997), pp.~2118--2127.

\bibitem{Burke2003}
{\sc J.~Burke, A.~S. Lewis, and M.~L. Overton}, {\em Optimization and
  pseudospectra, with applications to robust stability}, SIAM J. Matrix Anal.
  Appl., 25 (2003), pp.~80--104,
  \url{https://doi.org/10.1007/s10589-017-9898-5}.

\bibitem{Burke2006}
{\sc J.~V. Burke, D.~Henrion, A.~S. Lewis, and M.~L. Overton}, {\em {HIFOO} --
  {A MATLAB} package for fixed-order controller design and {$H_\infty$}
  optimization}, in Proc. 5th IFAC Syposium on Robust Control Design, Toulouse,
  France, Jul. 2006.

\bibitem{Burke2001}
{\sc J.~V. Burke, A.~S. Lewis, and M.~L. Overton}, {\em Optimal stability and
  eigenvalue multiplicity}, Found. Comput. Math., 1 (2001), pp.~205--225.

\bibitem{Burke2001a}
{\sc J.~V. Burke, A.~S. Lewis, and M.~L. Overton}, {\em Optimizing matrix
  stability}, Proc. Am. Math. Soc., 129 (2001), pp.~1635--1642.

\bibitem{Burke2002}
{\sc J.~V. Burke, A.~S. Lewis, and M.~L. Overton}, {\em Two numerical methods
  for optimizing matrix stability}, Linear Algebra Appl., 351-352 (2002),
  pp.~117--145.

\bibitem{Burke2003b}
{\sc J.~V. Burke, A.~S. Lewis, and M.~L. Overton}, {\em A nonsmooth, nonconvex
  optimization approach to robust stabilization by static output feedback and
  low-order controllers}, in Proc. 4th IFAC Syposium on Robust Control Design,
  Milan, Italy, June 2003.

\bibitem{Overton2003}
{\sc J.~V. Burke, A.~S. Lewis, and M.~L. Overton}, {\em Robust stability and a
  criss-cross algorithm for pseudospectra}, IMA J. Numer. Anal., 23 (2003),
  pp.~359--375.

\bibitem{Curtis2017}
{\sc F.~E. Curtis, T.~Mitchell, and M.~L. Overton}, {\em A
  {B}{F}{G}{S}-{S}{Q}{P} method for nonsmooth, nonconvex, constrained
  optimization and its evaluation using relative minimization profiles}, Optim.
  Method Softw., 32 (2017), pp.~148--181.

\bibitem{Gantmacher1959}
{\sc F.~R. Gantmacher}, {\em The theory of Matrices}, vol.~I,II, Chelsea
  Publishing Company, New York, NY, USA, 1959.

\bibitem{Guglielmi2013b}
{\sc N.~Guglielmi and G.~Lubich}, {\em Low-rank dynamics for computing extremal
  points of real pseudospectra}, SIAM J. Matrix Anal. Appl., 34 (2013),
  pp.~40--66.

\bibitem{Guglielmi2011}
{\sc N.~Guglielmi and M.~L. Overton}, {\em Fast algorithms for the
  approximation of the pseudospectral abscissa and pseudospectral radius of a
  matrix}, SIAM J. Matrix Anal. Appl., 32 (2011), pp.~1166--1192.

\bibitem{Gumussoy2009}
{\sc S.~Gumussoy and W.~Michiels}, {\em Computing $\mathcal{H}_\infty$ norms of
  time-delay systems}, in Proc. Joint 48th IEEE Conference on Decision and
  Control and 28th Chinese Control Conference, Shanghai, China, Dec. 2009,
  pp.~263--268.

\bibitem{Gurbuzbalaban2012}
{\sc M.~Gurbuzbalaban and M.~L. Overton}, {\em Some regularity results for the
  pseudospectral abscissa and pseudospectral radius of a matrix}, SIAM J.
  Optim., 22 (2012), pp.~281--285.

\bibitem{Kangal2018}
{\sc F.~Kangal, K.~Meerbergen, E.~Mengi, and W.~Michiels}, {\em A subspace
  method for large scale eigenvalue optimization}, SIAM J. Matrix Anal. Appl.,
  39 (2017), pp.~48--82.

\bibitem{Kreiss1962}
{\sc H.~O. Kreiss}, {\em {\"U}ber die stabilit{\"a}tsdefinition f{\"u}r
  differenzengleichungen die partielle differenzialgleichun- gen
  approximieren}, BIT, 2 (2012), pp.~153--181.

\bibitem{Kressner2018}
{\sc D.~Kressner, D.~Lu, and B.~Vandereycken}, {\em Subspace acceleration for
  the crawford number and related eigenvalue optimization problems}, SIAM J.
  Matrix Anal. Appl., 39 (2018), pp.~961--982.

\bibitem{Kressner2014}
{\sc D.~Kressner and B.~Vandereycken}, {\em Subspace methods for computing the
  pseudospectral abscissa and stability radius}, SIAM J. Matrix Anal. Appl., 35
  (2014), pp.~292--313.

\bibitem{Lehoucq1998}
{\sc R.~B. Lehoucq, D.~C. Sorensen, and C.~Yang}, {\em {A}{R}{P}{A}{C}{K} Users
  Guide: Solution of Large-Scale Eigenvalue Problems with Implicitly Restarted
  Arnoldi Methods}, SIAM, Philadelphia, PA, 1998.

\bibitem{Leibfritz2004}
{\sc F.~Leibfritz}, {\em \emph{{C}{O}{M}{P}}leib: \emph{{C}{O}}nstraint
  \emph{M}atrix-optimization \emph{{P}}roblem \emph{lib}rary -- a collection of
  test examples for nonlinear semidefinite programs, control system design and
  related problems}, tech. report, Department of Mathematics, University of
  Trier, 2004.

\bibitem{Lewis2008}
{\sc A.~S. Lewis and J.~P. Pang}, {\em Variational analysis of pseudospectra},
  SIAM J. Optim., 19 (2008), pp.~1048--1072.

\bibitem{Lu2017}
{\sc D.~Lu and B.~Vandereycken}, {\em Criss-cross type algorithms for computing
  the real pseudospectral abscissa}, SIAM J. Matrix Anal. Appl., 38 (2017),
  pp.~891--923.

\bibitem{Mengi2018}
{\sc E.~Mengi}, {\em Large-scale and global maximization of the distance to
  instability.}, SIAM J. Matrix Anal. Appl., 39 (2018), pp.~1776--1809.

\bibitem{Mengi2014}
{\sc E.~Mengi, E.~A. Yildirim, and M.~Kili\c{c}}, {\em Numerical optimization
  of eigenvalues of {H}ermitian matrix functions}, SIAM J. Matrix Anal. Appl.,
  35 (2014), pp.~699--724, \url{https://doi.org/10.1137/130933472},
  \url{http://dx.doi.org/10.1137/130933472}.

\bibitem{Nemirovski1993}
{\sc A.~Nemirovskii}, {\em Several {NP}-hard problems arising in robust
  stability analysis}, Math. Control Signals Syst., 6 (1993), pp.~99--105.

\bibitem{Hanso}
{\sc M.~L. Overton}, {\em {HANSO}: a hybrid algorithm for nonsmooth
  optimization}, 2009.
\newblock http://cs.nyu.edu/overton/software/hanso.

\bibitem{Rostami2015}
{\sc M.~W. Rostami}, {\em New algorithms for computing the real structured
  pseudospectral abscissa and the rel stability radius of large and sparse
  matrices}, SIAM J. Sci. Comput., 37 (2015), pp.~S447--S471.

\bibitem{Stewart1994}
{\sc G.~W. Stewart}, {\em Perturbation theory for rectangular matrix pencils},
  Linear Algebra Appl., 208-209 (1994), pp.~297--301.

\bibitem{Trefethen1999}
{\sc L.~N. Trefethen}, {\em Computation of pseudospectra}, Acta Numer., 8
  (1999), pp.~247--295.

\bibitem{Trefethen2005}
{\sc L.~N. Trefethen and M.~Embree}, {\em Spectra and Pseudospectra: The
  Behavior of Nonnormal Matrices and Operators}, Princetion University Press,
  Princeton, NJ, 2005.

\end{thebibliography}

\end{document}